\documentclass[12pt]{amsart}
\usepackage{epsf,amssymb, graphicx, multicol, amscd,times,stmaryrd}

\title{The embedded contact homology of sutured solid tori}

\author{Roman Golovko}

\address{D\'{e}partement de Math\'{e}matiques et de Statistique\\ Universit\'{e} de Montr\'{e}al\\\newline
Montr\'{e}al, QC H3T 1J4 \\Canada} \email{rgolovko@dms.umontreal.ca}
\urladdr{http://www.dms.umontreal.ca/~rgolovko}



\newtheorem{theorem}{Theorem}[section]
\newtheorem{lemma}[theorem]{Lemma}
\newtheorem{fact}[theorem]{Fact}
\newtheorem{claim}[theorem]{Claim}

\newtheorem{conjecture}[theorem]{Conjecture}
\newtheorem{proposition}[theorem]{Proposition}
\theoremstyle{remark}
\newtheorem{remark}[theorem]{Remark}
\theoremstyle{definition}
\newtheorem{definition}[theorem]{Definition}

\def\co{\colon\thinspace}

\numberwithin{equation}{subsection}

\begin{document}

\begin{abstract}
We calculate the relative versions of embedded contact homology,
contact homology and cylindrical contact homology of the sutured
solid torus $(S^1\times D^2,\Gamma)$, where $\Gamma$ consists of
$2n$ parallel longitudinal sutures.
\end{abstract}

\maketitle

\section{Introduction} \label{section: intro}
The embedded contact homology (ECH) of a closed, oriented
$3$-manifold with a contact form was introduced by Hutchings in
\cite{Hutchings, HutchingsSullivan, HutchingsTaubes,
HutchingsTaubes2} and is a variant of the symplectic field
theory~\cite{EliashbergGiventalHofer} of Eliashberg, Givental and
Hofer.  It is defined in terms of a contact form but is an invariant
of the underlying $3$-manifold. This invariance has been established
by Taubes in \cite{Taubes1, Taubes2} via the identification with
Seiberg-Witten Floer (co-)homology as defined by Kronheimer and
Mrowka~\cite{KronheimerMrowka2} and in particular implies the
Weinstein conjecture in dimension three. ECH is also conjecturally
isomorphic to Ozsv\'{a}th--Szab\'{o} Heegaard Floer homology defined
in \cite{OzsvathSzabo}. We would like to mention that Kutluhan, Lee
and Taubes, and independently Colin, Ghiggini and Honda  have
recently announced two different proofs of the isomorphism between
Hutchings's embedded contact homology and Heegaard Floer homology.

A natural condition to impose on a compact, oriented contact $(2m +
1)$-manifold $(M, \xi)$ with boundary is to require that $\partial
M$ be \emph{convex}, i.e., there is a contact vector field $X$
transverse to $\partial M$. To a transverse contact vector field $X$
we can associate the \emph{dividing set} $\Gamma=\Gamma_{X}\subset
\partial M$, namely the set of points $x\in \partial M$ such that
$X(x)\in \xi(x)$. By the contact condition, $(\Gamma, \xi \cap
T\Gamma)$ is a $(2m-1)$-dimensional contact submanifold of
$(M,\xi)$; the isotopy class of $(\Gamma, \xi \cap T\Gamma)$ is
independent of the choice of $X$. We will denote by $(M,\Gamma,
\xi)$ the contact manifold $(M, \xi)$ with convex boundary and
dividing set $\Gamma = \Gamma_{X}\subset \partial M$ with respect to
some transverse contact vector field $X$. Note that the actual
boundary condition we need is slightly different and is called a
\emph{sutured boundary condition}. (In the early 1980's, Gabai
developed the theory of sutured manifolds~\cite{Gabai}, which became
a powerful tool in studying $3$-manifolds with boundary.)  For the
moment we write $(M, \Gamma, \xi)$ to indicate either the convex
boundary condition or the sutured boundary condition.

It turns out that there is a way to generalize embedded contact
homology to sutured $3$-manifolds. This is possible by imposing a
certain convexity condition on the contact form. This construction
is completely described in the paper of Colin, Ghiggini, Honda and
Hutchings \cite{ColinGhigginiHondaHutchings}. Heegaard Floer
homology also admits a sutured version, namely the sutured Floer
homology (SFH) of Juh\'{a}sz~\cite{Juh'asz1, Juh'asz2}, which is an
invariant of sutured manifolds. Finally, Kronheimer and Mrowka in
\cite{KronheimerMrowka} introduced the sutured version of
Seiberg-Witten Floer homology.

Extending the conjectured equivalence of Heegaard Floer homology and
embedded contact homology, the following conjecture was formulated
in~\cite{ColinGhigginiHondaHutchings}:

\begin{conjecture}\label{conj_echsfh}
$SFH(-M,-\Gamma,\frak s_{\xi}+PD(h))\simeq ECH(M, \Gamma, \xi,h)$,
where $\frak s_{\xi}$ denotes the relative Spin$^c$-structure
determined by $\xi$ and $h\in H_{1}(M; \mathbb Z)$.
\end{conjecture}

In this paper, we construct sutured contact solid torus with $2n$
parallel longitudinal sutures, where $n\geq 2$, using the gluing
method of Colin, Ghiggini, Honda and Hutchings
\cite{ColinGhigginiHondaHutchings} and calculate the sutured
embedded contact homology of it. We apply the gluing method in such
a way that the constructed sutured solid torus is equipped with a
nondegenerate contact form satisfying the property that all closed
embedded Reeb orbits are noncontractible, define the same homology
class and have the same symplectic action. It turns out that for the
constructed sutured manifolds the sutured version of embedded
contact homology coincides with sutured Floer homology. The
corresponding calculation in sutured Floer homology has been done by
Juh\'{a}sz in \cite{Juh'asz3}. So far, this is the first series of
nontrivial examples where these two theories provide the same
answer.

\begin{theorem}\label{suturedECH}
Let $(S^1\times D^2, \Gamma)$ be a sutured manifold, where $\Gamma$
is a set of $2n$ parallel longitudinal curves and $n\geq 2$. Then
there is a contact form $\alpha$ which makes $(S^1\times D^2,
\Gamma)$ a sutured contact manifold and
\begin{align*}
ECH(S^1\times D^2,\Gamma, \alpha, J, h)\simeq \left \{
\begin{array}{ll}
\mathbb Z^{n-1 \choose h}, & \mbox{for}\ 0\leq h \leq n-1;\\
0, & \mbox{otherwise},
\end{array}
\right.
\end{align*}
where $h\in H_{1}(S^1\times D^2;\mathbb Z)$ and $J$ is an adapted
almost complex structure. Hence,
\begin{align*}
ECH(S^1\times D^2, \Gamma, \alpha, J) =\bigoplus\limits_{h\in
H_1(S^1\times D^2;\mathbb Z)}^{}ECH(S^1\times D^2,\Gamma, \alpha, J,
h) \simeq \mathbb Z^{2^{n-1}}.
\end{align*}
\end{theorem}

There is a Floer-type invariant of a closed, oriented contact
odd-dimensional manifold, called contact homology. Contact homology
was introduced by Eliashberg and Hofer and is a special case of the
symplectic field theory. In \cite{ColinGhigginiHondaHutchings},
Colin, Ghiggini, Honda and Hutchings generalized contact homology to
sutured manifolds.

For the sutured contact manifold from Theorem~\ref{suturedECH}, we
calculated the sutured versions of cylindrical contact homology and
contact homology.
\begin{theorem}\label{suturedCylContHom}
Let $(S^1\times D^2, \Gamma, \alpha)$ be a sutured contact manifold
from Theorem~\ref{suturedECH}. Then $HC^{cyl}(S^{1}\times
D^2,\Gamma,\alpha)$ is defined, is independent of the contact form
$\alpha$ for the given contact structure $\xi$ and almost complex
structure $J$,
\begin{align*}
HC^{cyl,h}(S^{1}\times D^2,\Gamma,\xi)&= \bigoplus\limits_{i\in
\mathbb Z} HC^{cyl,h}_{i}(S^1\times
D^2,\Gamma,\xi)\nonumber\\
&\simeq \left \{
\begin{array}{ll}
\mathbb Q^{n-1}, & \mbox{for}\ h\geq 1;\\
0, & \mbox{otherwise},
\end{array}
\right.
\end{align*}
and hence
\begin{align*}
HC^{cyl}(S^{1}\times D^2,\Gamma,\xi)= \bigoplus\limits_{h\geq
1}\bigoplus\limits_{i\in \mathbb Z} HC^{cyl,h}_{i}(S^1\times
D^2,\Gamma,\xi)\simeq \bigoplus\limits_{h\geq 1}\mathbb Q^{n-1},
\end{align*}
where $h$ is the homological grading and $i$ is the Conley-Zehnder
grading.
\end{theorem}

\begin{theorem}\label{suturedContHom}
Let $(S^1\times D^2, \Gamma, \alpha)$ be a sutured contact manifold
from Theorem~\ref{suturedECH}. Then $HC(S^{1}\times
D^2,\Gamma,\alpha)$ is defined, is independent of the contact form
$\alpha$ for the given contact structure $\xi$ and almost complex
structure $J$,
\begin{align*}
HC^{h}(S^{1}\times D^2,\Gamma,\xi)= \bigoplus\limits_{i\in \mathbb
Z} HC^{h}_{i}(S^1\times D^2,\Gamma,\xi)\simeq \mathbb Q^{\rho(n,h)}
\end{align*}
and hence
\begin{align*}
HC(S^{1}\times D^2,\Gamma,\xi)= \bigoplus\limits_{h\in \mathbb
Z}\bigoplus\limits_{i\in \mathbb Z} HC^{h}_{i}(S^1\times
D^2,\Gamma,\xi)\simeq \bigoplus\limits_{h\in \mathbb Z}\mathbb
Q^{\rho(n,h)},
\end{align*}
where $h$ is the homological grading, $i$ is the Conley-Zehnder
grading and $\rho(n,h)$ denotes the coefficient of $x^{h}$ in the
generating function $\prod^{\infty}_{s=1}(1+x^{s})^{n-1}$.

\end{theorem}

This paper is organized as follows: in Section~\ref{background} we
review definitions of embedded contact homology, cylindrical contact
homology and contact homology for sutured contact manifolds;
Section~\ref{construction} describes the construction of sutured
contact solid torus with $2n$ longitudinal sutures, where $n\geq 2$;
finally in Section~\ref{calculations} we calculate the relative
versions of embedded contact homology, cylindrical contact homology
and contact homology of the solid torus constructed in
Section~\ref{construction}.

\section*{Acknowledgements}
The author is deeply grateful to Ko Honda for his guidance, help and
support far beyond the call of duty. He also thanks Oliver Fabert,
Jian He and Michael Hutchings for helpful suggestions and interest
in his work. In addition, the author is grateful to the referee of
an earlier version of this paper for many valuable comments and
suggestions.

\section{Background}\label{background}
The goal of this section is to review definitions of embedded
contact homology, cylindrical contact homology and contact homology
for contact sutured manifolds. This section is essentially a summary
of \cite{ColinGhigginiHondaHutchings}.

\subsection{Sutured contact manifolds}
\label{section:sutcontman} In this section we repeat some
definitions from~\cite{ColinGhigginiHondaHutchings}.
\begin{definition}
A {\em Liouville manifold} (often also called a Liouville domain) is
a pair $(W, \beta)$ consisting of a compact, oriented
$2n$-dimensional manifold $W$ with boundary and a $1$-form $\beta$
on $W$, where $\omega = d\beta$ is a positive symplectic form on $W$
and the {\em Liouville vector field} $Y$ given by $i_{Y}(\omega) =
\beta$ is positively transverse to $\partial W$. It follows that the
$1$-form $\beta_{0} = \beta|_{\partial W}$ (this notation means
$\beta$ pulled back to $\partial W$) is a positive contact form with
kernel $\zeta$.
\end{definition}

\begin{definition}
A compact oriented $m$-dimensional manifold $M$ with boundary and
corners is a {\em sutured manifold} if it comes with an oriented,
not necessarily connected submanifold $\Gamma \subset \partial M$ of
dimension $m-2$ (called the {\em suture}), together with a
neighborhood $U(\Gamma)=[-1,0]\times[-1, 1]\times \Gamma$ of $\Gamma
= \{0\}\times \{0\}\times \Gamma$ in $M$, with coordinates
$(\tau,t)\in [-1,0]\times [-1,1]$, such that the following holds:
\begin{itemize}
\item[(1)] $U\cap \partial M=(\{0\}\times [-1,1]\times \Gamma)\cup([-1,0]\times
\{-1\} \times \Gamma)\cup ([-1,0]\times \{1\}\times
\Gamma)$;
\item[(2)] $\partial M \setminus (\{0\}\times (-1,1)\times \Gamma)=R_{-}(\Gamma)\sqcup
R_{+}(\Gamma)$, where the orientation of $\partial M$ agrees with
that of $R_{+}(\Gamma)$ and is opposite that of $R_{-}(\Gamma)$, and
the orientation of $\Gamma$ agrees with the boundary orientation of
$R_{+}(\Gamma)$;
\item[(3)] the corners of $M$ are precisely $\{0\}\times\{\pm 1\}\times
\Gamma$.
\end{itemize}
\end{definition}

The submanifold $\partial_{h} M=R_{+}(\Gamma)\cup R_{-}(\Gamma)$ is
called the {\em horizontal boundary} and $\partial_{v} M = \{ 0
\}\times [-1,1]\times \Gamma$ the {\em vertical boundary} of $M$.

\begin{definition}
Let $(M,\Gamma,U(\Gamma))$ be a sutured manifold. If $\xi$ is a
contact structure on $M$ (this means that $M$ is now
$2n+1$-dimensional), we say that $(M,\Gamma,U(\Gamma),\xi)$ is a
{\em sutured contact manifold} if $\xi$ is the kernel of a positive
contact 1-form $\alpha$ such that:
\begin{itemize}
\item[(1)] $(R_{+}(\Gamma), \beta_{+}=\alpha|_{R_{+}(\Gamma)})$
and $(R_{-}(\Gamma), \beta_{-}=\alpha|_{R_{-}(\Gamma)})$ are
Liouville manifolds;

\item[(2)] $\alpha=Cdt+\beta$ inside $U(\Gamma)$, where $C>0$ and $\beta$ is independent of
$t$ and does not have a $dt$-term;

\item[(3)] $\partial_{\tau} = Y_{\pm}$, where $Y_{\pm}$ is a Liouville vector field for $\beta_{\pm}$.
\end{itemize}
Such a contact form is said to be {\em adapted} to
$(M,\Gamma,U(\Gamma))$.
\end{definition}

\subsection{Completion of a sutured contact manifold}
\label{section:completion}

Let $(M,\Gamma, U(\Gamma),\xi)$ be a sutured contact manifold with
an adapted contact form $\alpha$. The form $\alpha$ is then given by
$Cdt +\beta_{\pm}$ on the neighborhoods $[1-\varepsilon,1]\times
R_{+}(\Gamma)$ and $[-1,-1 +\varepsilon]\times R_{-}(\Gamma)$ of
$R_{+}(\Gamma)=\{1\}\times R_{+}(\Gamma)$ and $R_{-}(\Gamma) =
\{-1\} \times R_{-}(\Gamma)$, where $t\in [-1,-1+\varepsilon]\cup
[1-\varepsilon,1]$ extends the $t$-coordinate on $U$. On $U$,
$\alpha = Cdt + \beta$, $\beta = \beta_{+} = \beta_{-}$, and
$\partial_{\tau}$ is a Liouville vector field $Y$ for $\beta$.

Following the procedure explained in
\cite{ColinGhigginiHondaHutchings} we can ``complete" $(M, \alpha)$
to a noncompact contact manifold $(M^{\ast}, \alpha^{\ast})$. We
first extend $\alpha$ to $[1,\infty)\times R_{+}(\Gamma)$ and
$(-\infty,-1]\times R_{-}(\Gamma)$ by taking $Cdt + \beta_{\pm}$ as
appropriate. The boundary of this new manifold is $\{0\}\times
\mathbb R\times \Gamma$. Notice that since $\partial_{\tau} = Y$,
the form $d\beta|_{[-1,0]\times \{t\}\times \Gamma}$ is the
symplectization of $\beta|_{\{0\}\times \{t\}\times \Gamma}$ in the
positive $\tau$-direction. We glue $[0,\infty)\times \mathbb R
\times \Gamma$ with the form $Cdt+e^{\tau}\beta_{0}$, where
$\beta_{0}$ is the pullback of $\beta$ to $\{0\}\times \{t\}\times
\Gamma$.

Let $M^{\ast}$ be the noncompact extension of $M$ described above
and $\alpha^{\ast}$ be the extension of $\alpha$ to $M^{\ast}$. For
convenience, we extend the coordinates $(\tau,t)$ -- so far defined
only on the ends of $M^{\ast}$ -- to functions on $M^{\ast}$ so that
$t(M)\subset [-1,1]$ and $\tau(M)\subset [-1,0]$. We then say that
$t
> 1$ corresponds the Top (T), $t < -1$ corresponds to the Bottom
(B), and $\tau>0$ corresponds to the Side (S). Let
$(\widehat{R_{\pm}(\Gamma)}, \widehat{\beta}_{\pm})$ be the
extension/completion of $(R_{\pm}(\Gamma), \beta_{\pm})$, obtained
by extending to (S).

\subsection{Reeb orbits and Conley-Zehnder index}
\label{section:reeborb_CZ}

Let $(M, \Gamma, U(\Gamma), \xi)$ be a sutured contact manifold with
an adapted contact form $\alpha$ and $(M^{\ast}, \alpha^{\ast})$ be
its completion.

The {\em Reeb vector field} $R_{\alpha^{\ast}}$ that is associated
to a contact form $\alpha^{\ast}$ is characterized by
\begin{align*}
\left \{ \begin{array}{ll} d\alpha^{\ast}(R_{\alpha^{\ast}},\cdot)=0;\\
\alpha^{\ast}(R_{\alpha^{\ast}})=1.
\end{array}
\right.
\end{align*}

A {\em Reeb orbit} is a closed orbit of the Reeb flow, i.e., a
smooth map $\gamma\co\mathbb R / T \mathbb Z \to M$ for some $T
> 0$ such that $\dot{\gamma}(t)=R_{\alpha^{\ast}}(\gamma(t))$.

\begin{remark}\label{reeborbincompl}
Every periodic orbit of $R_{\alpha^{\ast}}$ lies in $M$. Hence, the
set of periodic Reeb orbits of $R_{\alpha^{ \ast}}$ coincides with
the set of periodic Reeb orbits of $R_{\alpha}$.
\end{remark}

Consider Reeb orbit $\gamma$ passing through a point $x\in M$. The
linearization of the Reeb flow on the contact planes along $\gamma$
determines a linearized return map $P_{\gamma}\co \xi_x \to \xi_x$.
This linear map is symplectic and it does not depend on $x$ (up to
conjugation). The Reeb orbit $\gamma$ is {\em nondegenerate} if
$1\notin Spec(P_{\gamma})$.

Note that nondegeneracy can always be achieved by a small
perturbation, i.e., for any contact structure $\xi$ on $M$, there
exists a contact form $\alpha$ for $\xi$ such that all closed orbits
of $R_{\alpha}$ are nondegenerate.

For simplicity, we assume that all Reeb orbits of $R_{\alpha}$,
including multiply covered ones, are nondegenerate.

A Reeb orbit $\gamma$ is called {\em elliptic} or {\em positive}
(respectively {\em negative}) {\em hyperbolic} if the eigenvalues of
$P_{\gamma}$ are on the unit circle or the positive (resp. negative)
real line respectively.

If $\tau$ is a trivialization of $\xi$ over $\gamma$, we can then
define the Conley-Zehnder index. In $3$-dimensional situation this
is given explicitly as follows:
\begin{proposition}[\cite{Hutchings}]
If $\gamma$ is elliptic, then there is an irrational number $\phi\in
\mathbb R$ such that $P_{\gamma}$ is conjugate in $SL_2(\mathbb R)$
to a rotation by angle $2\pi\phi$, and
\begin{align*}
\mu_{\tau}(\gamma^k)=2\lfloor k\phi \rfloor + 1,
\end{align*}
where $2\pi\phi$ is the total rotation angle with respect to $\tau$
of the linearized flow around the orbit.

If $\gamma$ is positive (respectively negative) hyperbolic, then
there is an even (respectively odd) integer $r$ such that the
linearized flow around the orbit rotates the eigenspaces of
$P_{\gamma}$ by angle $\pi r$ with respect to $\tau$, and
\begin{align*}
\mu_{\tau}(\gamma^k)=kr.
\end{align*}
\end{proposition}

\subsection{Almost complex structure}
\label{section:almcomplstr} In this section we repeat some
definitions from Section 3.1 in \cite{ColinGhigginiHondaHutchings}.
\begin{definition} Let $(M, \xi)$ be a contact manifold with a
contact form $\alpha$ such that $\xi = ker(\alpha)$. An almost
complex structure $J$ on the symplectization $\mathbb R\times M$ is
{\em $\alpha$-adapted} if $J$ is $\mathbb R$-invariant; $J(\xi)=\xi$
with $d\alpha(v,Jv) > 0$ for nonzero $v\in \xi$; and
$J(\partial_s)=R_{\alpha}$, where $s$ denotes the $\mathbb
R$-coordinate and $R_{\alpha}$ is a Reeb vector field associated to
$\alpha$.
\end{definition}
\begin{definition}
Let $(W, \beta)$ be a Liouville manifold and $\zeta$ be the contact
structure given on $\partial W$ by $ker(\beta_{0})$, where
$\beta_{0} = \beta|_{\partial W}$. In addition, let $(\widehat{W},
\widehat{\beta})$ be the completion of $(W, \beta)$, i.e.,
$\widehat{W} = W\cup ([0,\infty)\times \partial W)$ and $\widehat
{\beta}|_{[0,\infty)\times \partial W} = e^{\tau}\beta_{0}$, where
$\tau$ is the $[0,\infty)$-coordinate. An almost complex structure
$J_{0}$ on $\widehat W$ is {\em $\widehat{\beta}$- adapted} if
$J_{0}$ is $\beta_{0}$-adapted on $[0,\infty)\times \partial W$; and
$d\beta (v,J_{0}v) > 0$ for all nonzero tangent vectors $v$ on $W$.
\end{definition}
\begin{definition}
Let $(M, \Gamma, U(\Gamma), \xi)$ be a sutured contact manifold,
$\alpha$ be an adapted contact form and $(M^{\ast}, \alpha^{\ast})$
be its completion. We say that an almost complex structure $J$ on
$\mathbb R\times M^{\ast}$ is {\em tailored to} $(M^{\ast},
\alpha^{\ast})$ if the following hold:
\begin{itemize}
\item[(1)] $J$ is $\alpha^{\ast}$-adapted;
\item[(2)] $J$ is $\partial_t$-invariant in a neighborhood of $M^{\ast} \setminus int(M)$;
\item[(3)] The projection of $J$ to $T \widehat{R_{\pm}(\Gamma)}$ is a $\widehat{\beta}_{\pm}$-adapted almost complex
structure $J_{0}$ on the completion $(\widehat{R_{+}(\Gamma)},
\widehat{\beta}_{+}) \bigsqcup (\widehat{R_{-}(\Gamma)},
\widehat{\beta}_{-})$ of the Liouville manifold $(R_{+}(\Gamma),
\beta_{+}) \bigsqcup (R_{-}(\Gamma), \beta_{-})$. Moreover, the flow
of $\partial_t$ identifies $J_0|_{\widehat{R_{+}(\Gamma)}\setminus
R_{+}(\Gamma)}$ and $J_{0}|_{\widehat{R_{-}(\Gamma)}\setminus
R_{-}(\Gamma)}$.
\end{itemize}
\end{definition}

\subsection{Sutured embedded contact homology}
First, let $M$ be a closed, oriented $3$-manifold, $\alpha$ be a
contact $1$-form on $M$ and let $J$ be an $\alpha$-adapted almost
complex structure on $\mathbb R\times M$. For simplicity, we assume
that all Reeb orbits of $R_{\alpha}$, including multiply covered
ones, are nondegenerate.

\begin{definition}
An {\em orbit set} is a finite set of pairs $a=\{ (\alpha_i, m_i)
\}$, where the $\alpha_i$'s are distinct embedded orbits of
$R_{\alpha}$ and the $m_i$'s are positive integers. The orbit set
$a$ is {\em admissible} if $m_{i} = 1$ whenever $\alpha_{i}$ is
hyperbolic. The homology class of $a$ is defined by
\begin{align*}
[a]:=\sum\limits_{i} m_i[\alpha_i] \in H_1(M; \mathbb Z).
\end{align*}
If $a=\{ (\alpha_i, m_i) \}$ and $b=\{ (\beta_j, n_j) \}$ are two
orbit sets with $[a]=[b]$, let $H_2(M, a, b)$ denote the set of
relative homology classes of $2$-chains $Z$ in $M$ with
\begin{align*}
\partial Z =
\sum\limits_{i} m_i\alpha_i - \sum\limits_{j} n_i\beta_j.
\end{align*}
\end{definition}

\begin{definition}
The {\em ECH chain complex} $C_{\ast}(M,\alpha,h)$ is a free
$\mathbb Z$-module with one generator for each admissible orbit set
$a$ with $[a]=h$.
\end{definition}

\begin{definition}
If $a = \{(\alpha_{i},m_{i})\}$ and $b = \{(\beta_{j}, n_{j})\}$ are
orbit sets with $[a] = [b]$, let $\mathcal M^{J}(a,b)$ denote the
moduli space of $J$-holomorphic curves $u$ with positive ends at
covers of $\alpha_{i}$ with total multiplicity $m_{i}$, negative
ends at covers of $\beta_{j}$ with total multiplicity $n_{j}$ , and
no other ends. Note that the projection of each $u \in \mathcal
M^{J}(a,b)$ to $M$ has a well-defined relative homology class
$[u]\in H_{2}(M,a,b)$. For $Z\in H_{2}(M,a,b)$ we then define
\begin{align*}
M^{J}(a,b,Z):=\{u\in M^{J}(a,b) |\ [u] = Z\}.
\end{align*}
\end{definition}

\begin{definition}
If $a=\{ (\alpha_i, m_i) \}$ is an orbit set, define the {\em
symplectic action}
\begin{align*}
\mathcal A(a):=\sum\limits_{i} m_i \int\limits_{\alpha_i} \alpha.
\end{align*}
\end{definition}

\begin{lemma}[\cite{HutchingsSullivan}]\label{le1_symplactfech}
For an adapted almost complex structure $J$, if $M^{J}(a, b)$ is
non-empty, then:
\begin{itemize}
\item[(1)] $\mathcal A(a)\geq \mathcal A(b)$.

\item[(2)] If $\mathcal A(a)= \mathcal A(b)$, then $a=b$ and every element of $\mathcal M^{J}(a,
b)$ maps to a union of trivial cylinders.

\end{itemize}

\end{lemma}

\begin{definition}
If $u\in \mathcal M^{J}(a, b, Z)$, define the {\em ECH index}
\begin{align*}
I(u)=I(a,b,Z)=c_1(\xi|_{Z},\tau)+Q_{\tau}(Z)+\sum\limits_{i}\sum\limits_{k=1}^{m_i}\mu_{\tau}(\alpha_i^k)-
\sum\limits_{j}\sum\limits_{k=1}^{n_j}\mu_{\tau}(\beta_j^k).
\end{align*}
\end{definition}
Here $Q_{\tau}(Z)$ denotes the relative intersection pairing, which
is defined in \cite{Hutchings}.

Any $J$-holomorphic curve $u\in M^{J} (a, b)$ can be uniquely
written as $u = u_{0}\cup u_{1}$, where $u_{0}$ and $u_{1}$ are
unions of components of $u$, each component of $u_{0}$ maps to an
$\mathbb R$-invariant cylinder, and no component of $u_{1}$ does.

\begin{proposition}[\cite{HutchingsTaubes}]\label{pr1_echind}
Suppose that $J$ is generic and $u = u_{0}\cup u_{1}\in \mathcal
M^{J} (a,b)$. Then:

\begin{itemize}

\item[(1)] $I(u)\geq 0$ with equality if and only if $u=u_{0}$.

\item[(2)] If $I(u)=1$, then $u$ contains one embedded component $u_{1}$ with
$ind(u_{1})=I(u_{1})=1$ and $u_{0}$ does not intersect $u_{1}$.
\end{itemize}

\end{proposition}

To fix the signs in the differential, fix some ordering of all the
embedded positive hyperbolic Reeb orbits in $M$.

Two curves $u$ and $u'$ in $\mathcal M^{J}(a,b,Z)/ \mathbb R$ are
equivalent if their embedded components $u_1$ and $u'_1$ are the
same up to translation, and if their other components cover each
embedded trivial cylinder $\mathbb R\times \gamma$ with the same
total multiplicity. The differential $\partial$ in $ECH$ counts $I =
1$ curves in $M^{J} (a, b)/\mathbb R$ where $a$ and $b$ are
admissible orbit sets. Such curves may contain multiple covers of
the $\mathbb R$-invariant cylinder $\mathbb R\times \gamma$ when
$\gamma$ is an elliptic embedded Reeb orbit. The differential
$\partial$ only keeps track of the total multiplicity of such
coverings for each $\gamma$. Finiteness of the count results from
the ECH compactness theorem \cite[Lemma 9.8]{Hutchings}. For the
sign of the count we refer to \cite{HutchingsTaubes2}.

Let $(M, \Gamma, U(\Gamma), \xi)$ be a sutured contact $3$-manifold
with an adapted contact form $\alpha$, $(M^{\ast}, \alpha^{\ast})$
be its completion and $J$ be an almost complex structure on $\mathbb
R\times M^{\ast}$ which is tailored to $(M^{\ast}, \alpha^{\ast})$.

The sutured embedded contact homology group $ECH(M, \Gamma, \alpha,
J)$ is defined to be the embedded contact homology of $(M^{\ast},
\alpha^{\ast}, J)$.

The following theorems have been proven by Colin, Ghiggini, Honda
and Hutchings in~\cite{ColinGhigginiHondaHutchings}:

\begin{theorem}[\cite{ColinGhigginiHondaHutchings}]
The ECH compactness theorem \cite[Lemma 9.8]{Hutchings} holds for
$J$-holomorphic curves in the symplectization of the completion of a
sutured contact $3$-manifold, provided that we choose the almost
complex structure $J$ on $\mathbb R \times M^{\ast}$ to be tailored
to $(M^{\ast}, \alpha^{\ast})$.
\end{theorem}

\begin{theorem}[\cite{ColinGhigginiHondaHutchings}]
Let $(M, \Gamma, U(\Gamma), \xi)$ be a sutured contact $3$-manifold
with an adapted contact form $\alpha$, $(M^{\ast}, \alpha^{\ast})$
be its completion and $J$ be an almost complex structure on $\mathbb
R\times M^{\ast}$ which is tailored to $(M^{\ast}, \alpha^{\ast})$.
Then the embedded contact homology group $ECH(M, \Gamma, \alpha, J)$
is defined.
\end{theorem}

\begin{remark}
Lemma~\ref{le1_symplactfech} and Proposition~\ref{pr1_echind} hold
for $J$-holomorphic curves in the symplectization of the completion
of a sutured contact manifold, provided that we choose the almost
complex structure $J$ on $\mathbb R \times M^{\ast}$ to be tailored
to $(M^{\ast}, \alpha^{\ast})$.
\end{remark}

Recall that embedded contact homology is an invariant of the
underlying closed, oriented $3$-manifold. Hence, it is natural to
expect the following:

\begin{conjecture}[\cite{ColinGhigginiHondaHutchings}]
The embedded contact homology group $ECH(M,\Gamma, \alpha, J)$ does
not depend on the choice of contact form $\alpha$, contact structure
$\xi = ker(\alpha)$, and almost complex structure $J$.
\end{conjecture}

\subsection{Sutured contact homology}
\label{section:sutcylconthom}

Let $(M, \Gamma, U(\Gamma), \xi)$ be a sutured contact manifold with
an adapted contact form $\alpha$, $(M^{\ast}, \alpha^{\ast})$ be its
completion and $J$ be an almost complex structure on $\mathbb
R\times M^{\ast}$ which is tailored to $(M^{\ast}, \alpha^{\ast})$.
For simplicity, we assume that all Reeb orbits of $R_{\alpha}$,
including multiply covered ones, are nondegenerate.

Let $\gamma$ be an embedded Reeb orbit. We are also interested in
the multiple covers $\gamma^{m}$ of $\gamma$, $m\geq 2$. There are 2
ways the Conley-Zehnder index of $\gamma^m$ can behave :

\begin{itemize}

\item[(1)] the parity of $\mu_{\tau}(\gamma^m)$ is the same for all $m \geq 1$.

\item[(2)] the parity for the even multiples $\mu_{\tau}(\gamma^{2k})$, $k \geq 1$,
disagrees with the parity for the odd multiples
$\mu_{\tau}(\gamma^{2k-1})$, $k\geq 1$.

\end{itemize}
In the second case, the even multiples $\gamma^{2k}$, $k \geq 1$,
are called {\em bad orbits}. An orbit that is not bad is called {\em
good}.

The {\em sutured contact homology algebra} $HC(M, \Gamma, \alpha,
J)$ is defined to be the contact homology of $(M^{\ast},
\alpha^{\ast}, J)$ in the following sense: The contact homology
chain complex $\mathcal A(\alpha, J)$ is the free supercommutative
$\mathbb Q$-algebra with unit generated by good Reeb orbits, where
the grading and the boundary map $\partial$ are defined in the usual
way (as in~\cite{EliashbergGiventalHofer}) with respect to the
$\alpha^{\ast}$-adapted almost complex structure $J$. The homology
of $ \mathcal A(\alpha, J)$ is the sutured contact homology algebra
$HC(M,\Gamma, \alpha, J)$.

We define the {\em sutured cylindrical contact homology group}
$HC^{cyl}(M, \Gamma, \alpha, J)$ to be the cylindrical contact
homology of $(M^{\ast}, \alpha^{\ast}, J)$. The cylindrical contact
homology chain complex $C(\alpha, J)$ is the $\mathbb Q$-module
freely generated by all good Reeb orbits, where the grading and the
boundary map $\partial$ are defined as in~\cite{Bourgeois} with
respect to the $\alpha^{\ast}$-adapted almost complex structure $J$.
The homology of $C(\alpha, J)$ is the sutured cylindrical contact
homology group $HC^{cyl}(M,\Gamma, \alpha, J)$.

For our calculations we will need the following fact which is a
consequence of Lemma~5.4 in
\cite{BourgeoisEliashbergHoferWysockiZehnder}:
\begin{fact}[\cite{BourgeoisEliashbergHoferWysockiZehnder}]\label{energyconthom}
Let $(M,\alpha)$ be a closed, oriented contact manifold with
nondegenerate Reeb orbits and
\begin{align*}
u=(a,f)\co (\dot{S},j)\to (\mathbb R\times M, J)
\end{align*} be a $J$-holomorphic curve in
$\mathcal M^{J}(\gamma; \gamma_{1},\dots,\gamma_{s})$, where
$\gamma$ and $\gamma_{i}$'s are all good Reeb orbits, $J$ is an
$\alpha$-adapted almost complex structure on $\mathbb R\times M$ and
$M^{J}(\gamma; \gamma_{1},\dots,\gamma_{s})$ is a moduli space of
$J$-holomorphic curves that we consider in contact homology. Then
\begin{align*}
\mathcal A(u)= \mathcal A(\gamma) - \sum\limits_{i=1}^{s} \mathcal
A(\gamma_{i}) = \int\limits_{\gamma}\alpha
-\sum\limits_{i=1}^{s}\int\limits_{\gamma_{i}}\alpha\geq 0
\end{align*}
with equality if and only if the image of $f$ is contained in a
trajectory of $R_{\alpha}$, i.e.,  $u$ maps to a trivial cylinder
over $\tilde{\gamma}$, where $\tilde{\gamma}$ is an embedded orbit
of $R_{\alpha}$, and hence $\gamma=\tilde{\gamma}^{k}$ for some $k$
and $\gamma_{i}=\tilde{\gamma}^{k_{i}}$ with $\sum_{i=1}^{s}
k_{i}=k$.
\end{fact}
In addition, we recall the following fact proven by Eliashberg and
Hofer:
\begin{fact}[\cite{Bourgeois}]\label{welldefcylconthom}
Let $(M,\alpha)$ be a closed, oriented contact manifold with
nondegenerate Reeb orbits. Let $C_{i}^{h}(M,\alpha)$ be the
cylindrical contact homology complex, where $h$ is a homotopy class
of Reeb orbits and $i$ corresponds to the Conley-Zehnder grading. If
there are no contractible Reeb orbits, then for every free homotopy
class $h$
\begin{itemize}

\item[($1$)] $\partial^2=0$;

\item[($2$)] $H(C^{h}_{\ast}(M,\alpha),\partial)$ is independent of the contact form $\alpha$ for $\xi$,
the almost complex structure $J$ and the choice of perturbation for
the moduli spaces.
\end{itemize}
\end{fact}

Now we remind the following theorems which have been proven by
Colin, Ghiggini, Honda and Hutchings
in~\cite{ColinGhigginiHondaHutchings}:
\begin{theorem}[\cite{ColinGhigginiHondaHutchings}]
The SFT compactness theorem \cite[Theorem
10.1]{BourgeoisEliashbergHoferWysockiZehnder} holds for
$J$-holomorphic curves in the symplectization of the completion of a
sutured contact manifold, provided that we choose the almost complex
structure $J$ on $\mathbb R \times M^{\ast}$ to be tailored to
$(M^{\ast}, \alpha^{\ast})$.
\end{theorem}

\begin{theorem}[\cite{ColinGhigginiHondaHutchings}]\label{welldefconthom}
Let $(M, \Gamma, U(\Gamma), \xi)$ be a sutured contact $3$-manifold
with an adapted contact form $\alpha$, $(M^{\ast}, \alpha^{\ast})$
be its completion and $J$ be an almost complex structure on $\mathbb
R\times M^{\ast}$ which is tailored to $(M^{\ast}, \alpha^{\ast})$.
Then the contact homology algebra $HC(M, \Gamma, \xi)$ is defined
and independent of the choice of contact $1$-form $\alpha$ with $ker
(\alpha) = \xi$, adapted almost complex structure $J$, and abstract
perturbation.
\end{theorem}

\begin{remark}\label{actchomcchom}
Fact~\ref{energyconthom} and Fact~\ref{welldefcylconthom} hold for
$J$-holomorphic curves in the symplectization of the completion of a
sutured contact manifold, provided that we choose the almost complex
structure $J$ on $\mathbb R \times M^{\ast}$ to be tailored to
$(M^{\ast}, \alpha^{\ast})$.
\end{remark}

Note that Theorem~\ref{welldefconthom} and Remark~\ref{actchomcchom}
rely on the assumption that the machinery, needed to prove the
analogous properties for contact homology and cylindrical contact
homology in the closed case, works.

\subsection{Gluing sutured contact manifolds}
\label{section:glsutcontman}

Now we briefly describe the procedure of gluing sutured contact
manifolds, together with compatible Reeb vector fields which was
first described by Colin and Honda in \cite{ColinHonda} and
generalized in \cite{ColinGhigginiHondaHutchings}.

\begin{remark}
In \cite{Gabai}, Gabai defined the notion of a sutured manifold
decomposition for sutured 3-manifolds, which is the inverse
construction of the sutured gluing.
\end{remark}

Let $(M',\Gamma',U(\Gamma'),\xi')$ be a sutured contact $3$-manifold
with an adapted contact form $\alpha'$. We denote by $\pi$ the
projection along $\partial_{t}$ defined on $U(\Gamma')$. If we think
of $[-1, 0]\times \Gamma'$ as a subset of $R_{+}(\Gamma')$ (resp.
$R_{-}(\Gamma')$), then we denote the projection by $\pi_{+}$ (resp.
$\pi_{-}$). By definition, the horizontal components
$(R_{\pm}(\Gamma'),\beta'_{\pm} =\alpha'|_{R_{\pm}(\Gamma')})$ are
Liouville manifolds. We denote by $Y'_{\pm}$ their Liouville vector
field. The contact form $\alpha'$ is $dt+\beta'_{\pm}$ on the
neighborhoods $R_{+}(\Gamma')\times[1-\varepsilon,1]$ and
$R_{-}(\Gamma')\times [-1,-1 + \varepsilon]$ of
$R_{+}(\Gamma')=R_{+}(\Gamma')\times\{1\}$ and
$R_{-}(\Gamma')=R_{-}(\Gamma')\times\{-1\}$. In addition, we may
assume without loss of generality that the Reeb vector field
$R_{\alpha'}$ is given by $\partial_{t}$ on $U(\Gamma')$.

Take a $2$-dimensional submanifolds $P_{\pm}\subset
R_{\pm}(\Gamma')$ such that $\partial P_{\pm}$ is the union of
$(\partial P_{\pm})_{\partial}\subset \partial R_{\pm}(\Gamma')$,
$(\partial P_{\pm})_{int}\subset int(R_{\pm}(\Gamma'))$ and
$\partial P_{\pm}$ is positively transversal to the Liouville vector
field $Y'_{\pm}$ on $R_{\pm}(\Gamma')$.

Whenever we refer to $(\partial P_{\pm})_{int}$ and $(\partial
P_{\pm})_{\partial}$, we assume that closures are taken as
appropriate. Moreover we make the assumption that $\pi((\partial
P_{-})_{\partial})\cap \pi((\partial P_{+})_{\partial})=\emptyset$.

Let $\varphi$ be a diffeomorphism which sends
$(P_{+},\beta'_{+}|_{P{+}})$ to $(P_{-},\beta'_{-}|_{P_{-}})$ and
takes $(\partial P_{+})_{int}$ to $(\partial P_{-})_{\partial}$ and
$(\partial P_{+})_{\partial}$ to $(\partial P_{-})_{int}$. We will
refer to the triple $(P_{+}, P_{-}, \varphi)$ as the {\em gluing
data}. For the purposes of gluing, we only need
$\beta'_{+}|_{P_{+}}$ and $\varphi^{\ast}(\beta'_{-}|_{P_{-}})$ to
match up on $\partial P_{+}$, since we can linearly interpolate
between primitives of positive area forms on a surface.

Topologically, we construct the sutured manifold $(M, \Gamma)$ from
$(M', \Gamma')$ and the gluing data $(P_{+}, P_{-}, \varphi)$ as
follows: Let $M=M'/\sim$, where
\begin{itemize}
\item $x\sim \varphi(x)$ for all $x\in P_{+}$;

\item $x\sim x'$ if $x, x'\in \pi^{-1}(\Gamma')$ and $\pi(x)=\pi(x')\in
\Gamma'$.
\end{itemize}
Then
\begin{align*}
R_{\pm}(\Gamma)=\frac{\overline{R_{\pm}(\Gamma')\setminus
P_{\pm}}}{(\partial P_{\pm})_{int}}\sim \pi_{\pm}((\partial
P_{\mp})_{\partial})
\end{align*}
and
\begin{align*}
\Gamma=\frac{\overline{\Gamma'\setminus \pi (\partial P_{+}\sqcup
\partial P_{-})}}{\pi((\partial P_{+})_{int}\cap (\partial
P_{+})_{\partial})} \sim \pi((\partial P_{-})_{int}\cap (\partial
P_{-})_{\partial}).
\end{align*}

In dimension $3$, for the purposes of studying holomorphic curves,
we want to stretch in $t$-direction. In higher dimensions, one needs
to stretch in both $\tau$- and $t$-directions. The construction
depends on the parameter $N$, where $N$ is a stretching parameter in
$t$-direction, and the resulting glued-up sutured contact manifold
is written as $(M_{N}, \Gamma_{N}, U(\Gamma_{N}),
\xi_{N}=ker(\alpha_{N}))$.

Let $M^{(0)} = M^{(0)}_{N}$ (we will suppress $N$ to avoid
cluttering the notation) be the manifold obtained from the
completion $(M')^{\ast}$ by removing the Side (S), i.e.,
\begin{align*}
M^{(0)}=M'\cup (R_{+}(\Gamma')\times[1;\infty))\cup
(R_{+}(\Gamma')\times (-\infty;-1]).
\end{align*}

Then construct $M^{(1)}$ from
\begin{align*}
M^{(0)} \setminus ((P_{+}\times [N, \infty))\cup (P_{-}\times
(-\infty,-N])),
\end{align*}
by taking closures and identifying:
\begin{itemize}
\item $P_{+}\times \{N\}$ with $P_{-}\times \{-N\}$;
\item $(\partial P_{+})_{int} \times [N,\infty)$ with $(\partial P_{-})_{\partial}\times [-N, \infty)$;
\item $(\partial P_{+})_{\partial} \times (-\infty, N]$ with $(\partial P_{-})_{int} \times (-\infty,-N]$;
\end{itemize}
all via the identification $(x,t)\mapsto (\varphi(x), t-2N)$.

Next we take $N'\gg 0$ and truncate the Top and Bottom of $M^{(1)}$
to obtain the (compact) sutured manifold $(M^{(2)}, \Gamma^{(2)},
U(\Gamma^{(2)}))$ so that $M^{(2)}$ contains
\begin{align*}
M'\cup ((\overline{R_{+}(\Gamma')\setminus P_{+}})\times [1,
N'])\cup ((\overline{R_{-}(\Gamma')\setminus P_{-}})\times
[-N',-1]),
\end{align*}
the Reeb vector field is transverse to the horizontal boundary, and
the vertical boundary $E$ is foliated by interval Reeb orbits with
fixed action $\geq 3N'$. Attaching  $V = [0, \tau_{0}]\times E$ to
$M^{(2)}$ for some specific $\tau_{0}$ gives us $(M_{N}, \Gamma_{N},
U(\Gamma_{N}))$. The horizontal boundary which is positively (resp.
negatively) transverse to $R$ will be called $R_{+}(\Gamma_{N})$
(resp. $R_{-}(\Gamma_{N})$). For more details we refer to
\cite{ColinGhigginiHondaHutchings}.

\section{Construction}
\label{construction}

In this section we construct a sutured contact solid torus with $2n$
longitudinal sutures, where $n\geq 2$.
\subsection{Gluing map}
\label{section:gl_map}

Now we construct $H\in C^{\infty}(\mathbb R^2)$. The flow of the
Hamiltonian vector field associated to $H$ will play a role of
gluing map when we will apply the gluing construction described in
Section~\ref{section:glsutcontman} to the sutured contact solid
cylinder constructed in Section~\ref{section:gluing}.

We fix $x\in \mathbb R^2$ and consider $H_{sing}\co\mathbb R^2\to
\mathbb R$ given by $H_{sing}=\mu r^2 \cos(n\theta)$ in polar
coordinates about $x$, where $\mu>0$ and $n\in \mathbb N$. Note that
$H_{sing}$ is singular only at $x$. We obtain $H\in
C^{\infty}(\mathbb R^2)$ from $H_{sing}$ by perturbing $H_{sing}$ on
a small disk $D(r_{sing})$ about $x$ in such a way that $H$ has
$n-1$ nondegenerate saddle points and interpolates with no critical
points with $H_{sing}$ on $D(r_{sing})$. In other words,
$H=H_{sing}$ on $\mathbb R^2\setminus D(r_{sing})$. For the level
sets of $H_{sing}$ and $H$ in the case $n=3$ we refer to
Figure~\ref{levelset3}.

The construction of $H$ was initially described by Cotton-Clay as a
construction of a Hamiltonian function whose time-$1$ flow is a
symplectic smoothing of the singular representative of pseudo-Anosov
map in a neighborhood of a singular point with $n$ prongs
in~\cite{Cotton-Clay}.

Since some of the properties of $H$ described in~\cite{Cotton-Clay}
will be important for further discussion, we will state them in the
next remark.

\begin{remark}\label{re1_closedlevsets}
We can choose $H$ in such a way that it satisfies the following
properties:
\begin{itemize}
\item[(1)] $H$ can be written as $x\sin(\pi y)$ in some coordinates $(x,y)$ in a connected neighborhood
containing its critical points;

\item[(2)] there are no components of level sets of $H$ which are
circles;

\item[(3)] there is an embedded curve which is a component of one of the level
curves of $H$ and connects all the saddle points of $H$. We call
this embedded curve $\gamma$.
\end{itemize}
\end{remark}

For the detailed construction of $H$ we refer to Lemmas 3.4 and 3.5
in~\cite{Cotton-Clay}.

\begin{figure}[t]
\begin{center} {\includegraphics [height=200pt]
    {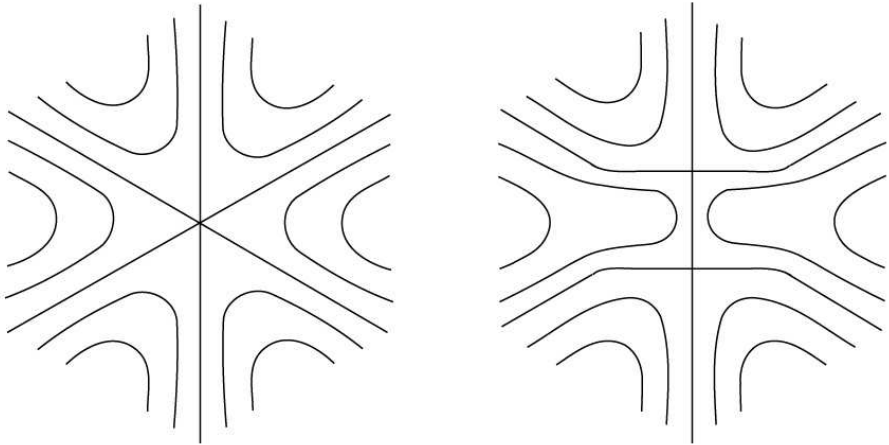}}
\end{center}
\caption{The level sets of $H_{sing}$ (left) and the level sets of
$H$ (right) in the case $n=3$.} \label{levelset3}
\end{figure}

\begin{lemma}\label{le1_prarfm}
Let $s$ be a saddle point of $H$. Then there are coordinates $(x,y)$
about $s$ such that $H=axy$ for $a>0$.
\end{lemma}
\begin{proof}
First observe that from Remark~\ref{re1_closedlevsets} it follows
that $H(s)=0$. By Morse lemma, there are coordinates $(x',y')$ about
$s$ such that $H=H(s)-x'^2+y'^2$. Given $H(s)=0$, we can write
\begin{align*}
H=-x'^2+y'^2=a\left (\frac{y'}{\sqrt{a}}-\frac{x'}{\sqrt{a}}\right
)\left (\frac{y'}{\sqrt{a}}+\frac{x'}{\sqrt{a}}\right ).
\end{align*}
Now let
\begin{align*}x=\frac{y'}{\sqrt{a}}+\frac{x'}{\sqrt{a}} \quad
\mbox{and} \quad
y=\frac{y'}{\sqrt{a}}-\frac{x'}{\sqrt{a}}.\end{align*} Clearly $x$
and $y$ satisfy the statement of the lemma. In addition, the
orientation of the pair $(x,y)$ coincides with the orientation of
$(x',y')$.
\end{proof}
Let $U_{k}$ be a neighborhood of the $k$-th saddle point $p_{k}$ of
$H$ from Lemma~\ref{le1_prarfm}.

\subsection{Contact form}
\label{section:cont_form}

\begin{claim}
If $(M,\omega)$ is an exact symplectic manifold, i.e.,
$\omega=d\beta$, then the flow $\varphi_{X_{H}}^{t}$ of a
Hamiltonian vector field $X_{H}$ consists of exact symplectic maps,
i.e.,
\begin{align*}
(\varphi_{X_{H}}^{t})^{\ast}\beta - \beta = df_{t}\quad \mbox{for
some function}\ f_{t}.
\end{align*}
\end{claim}
\begin{proof}
Since $\varphi^{0}=id$,
\begin{align*}
(\varphi_{X_{H}}^{t})^{\ast}\beta - \beta = \int\limits^{t}_{0}
\frac{d}{ds}(\varphi_{X_{H}}^s)^{\ast}\beta ds.
\end{align*}
Since by definition $i_{X_H}\omega=-dH$, the integrand is equal to
\begin{align*}
(\varphi_{X_{H}}^s)^{\ast}L_{X_{H}}\beta & =
(\varphi^s)^{\ast}(i_{X_H}d\beta + di_{X_{H}}\beta) =
(\varphi_{X_{H}}^s)^{\ast}(-dH+d\beta(X_{H})) \\
& = d(\varphi_{X_{H}}^s)^{\ast}(-H+\beta(X_{H})) =
d([-H+\beta(X_{H})]\circ \varphi_{X_{H}}^{s}).
\end{align*}
Thus
\begin{align*}
f_t=\int\limits^{t}_{0}(-H+\beta(X_{H}))\circ \varphi_{X_{H}}^{s}ds.
\end{align*}
\end{proof}

Notice that our definition of $X_{H}$ is slightly different from the
standard one; usually $X_{H}$ is defined by $i_{X_H}\omega=dH$.

Note that the condition that $\beta(X_{H})=H$ is equivalent to the
condition that $L_{X_{H}}\beta=0$.

\begin{remark}\label{re1_funcanalcon}
Let $f:=f_1=\int^{1}_{0}(-H+\beta(X_{H}))\circ
\varphi_{X_{H}}^{s}ds$. In addition, let $S\subset M$ be a region
such that $\beta(X_{H})=H$ on $S$ and $S':=\{ s\in S :
\varphi_{X_{H}}^{t}(s)\in S\ \forall  t\in[0,1] \}$. Then
$f|_{S'}=0$ and $(\varphi^{1}_{X_{H}})^{\ast}(\beta)=\beta$ on $S'$.
\end{remark}

In the next two lemmas we construct a $1$-form $\beta$ on $\mathbb
R^2$ with $d\beta>0$ and show that $\beta$ is ``adapted'' to $H$,
i.e., $\varphi_{X_{H}}^{\ast}\beta=\beta$ near the saddle points of
$H$ and on the region far enough from $D(r_{sing})$, where $X_{H}$
is a Hamiltonian vector field with respect to $d\beta$ and
$\varphi_{X_{H}}$ is the time-$1$ map of the flow of $X_{H}$. The
condition that $\beta (X_{H})=H$ and Remark~\ref{re1_funcanalcon}
will play a crucial role when we will compare
$\varphi_{X_{H}}^{\ast}\beta$ and $\beta$.

\begin{lemma}\label{le1_formanalcon}
There exists a 1-form $\beta$ on $\mathbb R^2$ satisfying the
following:
\begin{itemize}
\item[(1)] $d\beta > 0$;

\item[(2)] the singular foliation given by $ker(\beta)$ has isolated
singularities and no closed orbits;

\item[(3)] the elliptic points of the singular foliation of $\beta$ are
the saddle points of $H$; $\beta=\frac{\varepsilon}{2}(xdy-ydx)$ on
$U_{k}$ with respect to the coordinates from Lemma~\ref{le1_prarfm},
where $k\in \{1,\dots,n-1\}$ and $\varepsilon$ is a small positive
real number;
\item[(4)] $\beta=\frac{1}{2}r^2d\theta$ on $\mathbb
R^2\setminus D(r_{sing})$ with respect to the polar coordinates
whose origin is at the center of $D(r_{sing})$;
\item[(5)] the hyperbolic points of the singular foliation of $\beta$ are
located on $\gamma$, outside of $U_{k}$'s and distributed in such a
way that between each two closest elliptic points there is exactly
one hyperbolic point.
\end{itemize}
\end{lemma}

\begin{proof}
Consider a singular foliation $\mathcal F$ on $\mathbb R^2$ which
satisfies the following:

\begin{itemize}
\item[(1)] $\mathcal F$ is Morse-Smale and has no closed orbits.

\item[(2)] The singular set of $\mathcal F$ consists of elliptic points and hyperbolic points.
The elliptic points are the saddle points of $H$. The hyperbolic
points are located on $\gamma$ and distributed in such a way that
between each two closest elliptic points there is exactly one
hyperbolic point. In addition, the hyperbolic points are outside of
$U_{k}$'s.

\item[(3)] $\mathcal F$ is oriented, and for one choice of orientation the flow is transverse to
and exits from $\partial D(r_{sing})$.
\end{itemize}

\begin{figure}[t]
\begin{center} {\includegraphics [height=200pt]
    {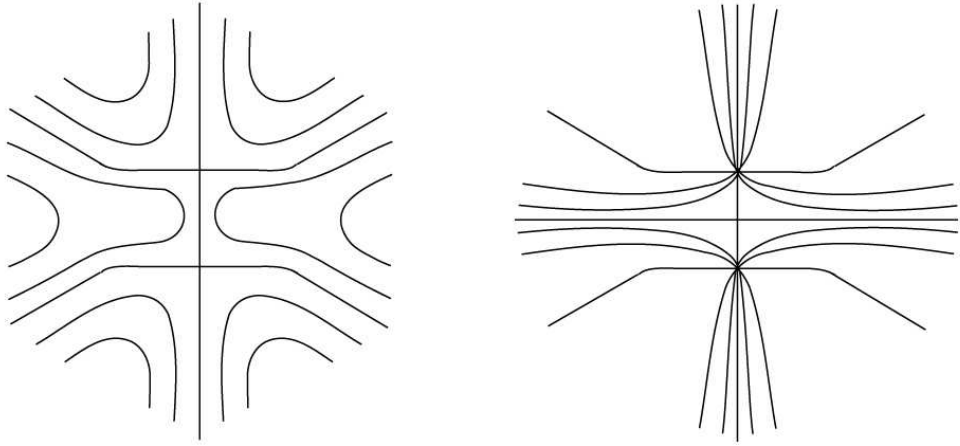}}
\end{center}
\caption{The level sets of $H$ (left) and the characteristic
foliation of $\beta$ (right) in the case $n=3$}\label{charfol}
\end{figure}

Next, we modify $\mathcal F$ near each of the singular points so
that $\mathcal F$ is given by $\beta_{0} = \frac{1}{2}(xdy-ydx)$ on
$U_{k}$ with respect to the coordinates from Lemma~\ref{le1_prarfm},
and $\beta_0 = 2xdy+ydx$ near a hyperbolic point. In addition, on
$\mathbb R^2\setminus D(r_{sing})$,
$\beta_{0}=\frac{1}{2}r^2d\theta$ with respect to the polar
coordinates whose origin is at the center of $D(r_{sing})$. Finally,
we get $\mathcal F$ given by $\beta_{0}$, which satisfies $d\beta_0
> 0$ near the singular points and on $\mathbb R^2\setminus D(r_{sing})$. Now let $\beta
= g\beta_0$, where $g$ is a positive function with $dg(X) \gg 0$
outside of $(\cup^{n-1}_{k=1}U_{k})\cup(\mathbb R^2\setminus
D(r_{sing}))$, $g|_{\cup^{n-1}_{k=1}U_{k}}=\varepsilon$,
$g|_{\mathbb R^2\setminus D(r_{sing})}=1$ and $X$ is an oriented
vector field for $\mathcal F$ (nonzero away from the singular
points). Since $d\beta = dg\wedge \beta_{0} + g\wedge d\beta_{0}$,
$dg (X) \gg 0$ guarantees that $d\beta
> 0$. Here $\varepsilon$ is a small positive real number.
\end{proof}

\begin{remark}\label{re1_formandfunc}
From the previous lemma we get $\beta$ defined on $\mathbb R^2$ with
the following properties:

\begin{itemize}

\item[($i$)] $d\beta >0$ on $\mathbb R^2$.

\item[($ii$)] $\beta=\frac{\varepsilon}{2}(xdy-ydx)$ and $H=axy$ on $U_{k}$ for $k=1,\dots,n-1$.
In other words, the saddle points of $H$ are exactly the elliptic
singularities of $\beta$.

\item[($iii$)] $\beta=\frac{1}{2}r^2d\theta$ and $H=\mu r^2 \cos(n\theta)$ on
$\mathbb R^2\setminus D(r_{sing})$.
\end{itemize}
\end{remark}
For the comparison of the level sets of $H$ with the singular
foliation of $\beta$ in the case $n=3$ we refer to
Figure~\ref{charfol}.

\begin{lemma}\label{le2_formanalcon}
Let $\beta$ be a $1$-form from Lemma~\ref{le1_formanalcon}. The
Hamiltonian vector field $X_{H}$ of $H$ with respect to the area
form $d\beta$ satisfies $\beta(X_{H})=H$ on $(\cup_{k=1}^{n-1}
U_{k})\cup ( \mathbb R^2\setminus D(r_{sing}))$.
\end{lemma}
\begin{proof}
First we work on $U_{k}$, where $k\in \{1,\dots,n-1\}$. From
Remark~\ref{re1_formandfunc} it follows that
$\beta=\frac{\varepsilon}{2}(xdy-ydx)$ and $H=axy$ on $U_{k}$. Let
$X_{H}$ be a Hamiltonian vector field defined by $i_{X_{H}}
d\beta=-dH$. We show that
\begin{align*}X_{H}=-\frac{ax}{\varepsilon}
\frac{\partial}{\partial x}+\frac{ay}{\varepsilon}
\frac{\partial}{\partial y}
\end{align*}
is a solution of the equation
\begin{align}\label{e1_formanconstr}
\beta(X_{H})=H
\end{align}
on $U_{k}$. We calculate
\begin{align*}
i_{X_{H}}( d\beta)=(-\frac{ax}{\varepsilon}\frac{\partial}{\partial
x}+\frac{ay}{\varepsilon}\frac{\partial}{\partial y}) \lrcorner
(\varepsilon dx\wedge dy)=-axdy-aydx=-dH
\end{align*}
and
\begin{align*}
\beta(X_{H})=\frac{\varepsilon}{2}(xdy-ydx)\left
(-\frac{ax}{\varepsilon}\frac{\partial}{\partial
x}+\frac{ay}{\varepsilon}\frac{\partial}{\partial y}\right )=axy=H.
\end{align*}

Next, by Remark~\ref{re1_formandfunc}, $\beta=\frac{1}{2}r^2d\theta$
and $H=\mu r^2\cos(n\theta)$ on $\mathbb R^2\setminus D(r_{sing})$.

As in the previous case, we show that
\begin{align*}
X_{H}=n\mu r\sin(n\theta)\frac{\partial}{\partial r}+2\mu
\cos(n\theta)\frac{\partial}{\partial \theta}
\end{align*}
is a solution of Equation~(\ref{e1_formanconstr}) on $\mathbb
R^2\setminus D(r_{sing})$. We calculate
\begin{align*}
i_{X_{H}}(d\beta)&=(n\mu r \sin(n\theta)\partial_{r}+2\mu
\cos(n\theta)\partial_{\theta})\lrcorner(rdr\wedge d\theta)\\
&=-2\mu r \cos(n\theta)dr+n\mu r^2\sin(n\theta)d\theta=-dH
\end{align*}
and
\begin{align*}
\beta(X_{H})=\left (\frac{1}{2}r^2d\theta \right ) \left (n\mu r
\sin(n\theta)\frac{\partial}{\partial r}+2\mu
\cos(n\theta)\frac{\partial}{\partial \theta}\right )=\mu r^2
\cos(n\theta)=H.
\end{align*}
\end{proof}
Let $\varphi^{s}_{X_{H}}$ be the time-$s$ flow of $X_{H}$. Consider
\begin{align*}
&S:=\{ x\in  \mathbb R^2\setminus D(r_{sing}):
\varphi_{X_{H}}^{s}(x)\in \mathbb R^2\setminus D(r_{sing})\
\forall s\in[0,1] \}\quad \mbox{and}\\
&V_{k}:=\{ x\in  U_{k}: \varphi_{X_{H}}^{s}(x)\in U_{k}\ \forall
s\in[0,1] \}.
\end{align*}
Since the saddle points of $H$ are the fixed points of
$\varphi^{s}_{X_{H}}$ for $s\in [0,1]$, $V_{k}$ contains an open
neighborhood about the $k$-th saddle point of $H$. In addition, note
that $S$ contains the level sets of $H$ which do not intersect
$D(r_{sing})$, and $(\mathbb R^2\setminus D(R))$, where $R\gg
r_{sing}$. For ease of notation, we write $\varphi_{X_{H}}$ instead
of $\varphi^{1}_{X_{H}}$.
\begin{remark}\label{re1_analconstr}
Lemma~\ref{le2_formanalcon} implies that
$\varphi_{X_{H}}^{\ast}(\beta)=\beta$ on $S\cup
(\cup^{n-1}_{k=1}V_{k})$.
\end{remark}

\begin{remark}\label{re1_phiperpoints}
From Remark~\ref{re1_closedlevsets} and the fact that the flow of
$X_{H}$ preserves the level sets of $H$ it follows that
$\{p_{k}\}^{n-1}_{k=1}$ is the set of periodic points of
$\varphi_{X_{H}}$ on $\mathbb R^2$.
\end{remark}

In the next lemma we construct a contact form $\alpha$ on $D^2\times
[-1,1]$ such that $R_{\alpha}$ has vertical trajectories.

\begin{lemma}\label{le1_almvertical}
Let $\beta_0$ and $\beta_1$ be two $1$-forms on $D^2$ such that
$\beta_0=\beta_1$ in a neighborhood of $\partial D^2$ and
$d\beta_0=d\beta_1=\omega>0.$ Then there exists a contact $1$-form
$\alpha$ with Reeb vector field $R_{\alpha}$ on $[-1,1] \times D^2$
with coordinates $(t,x)$, where t is a coordinate on $[-1,1]$ and
$x$ is a coordinate on $D^2$, with the following properties:
\begin{itemize}
\item[(1)] $\alpha=dt+\varepsilon\beta_{0}$ in a neighborhood of $\{-1\}\times
D^2$;
\item[(2)] $\alpha=dt+\varepsilon\beta_{1}$ in a neighborhood of $\{1\}\times
D^2$;
\item[(3)] $R_{\alpha}$ is collinear to $\frac{\partial}{\partial
t}$ on $[-1,1]\times D^2$;
\item[(4)] $R_{\alpha}=\frac{\partial}{\partial t}$ in a
neighborhood of $[-1,1]\times \partial D^2$.
\end{itemize}
Here $\varepsilon$ is a small positive number.
\end{lemma}
\begin{proof}
Since $D^2$ is simply connected and $\omega=d\beta_0=d\beta_1>0$,
there exists a function $h \in C^{\infty}(D^2)$ such that $\beta_1 -
\beta_0 = dh$. Let $\chi_{0}\co [-1,1]\to [0,1]$ be a smooth map for
which $\chi_{0}(t)=0$ for $-1\leq t \leq-1+\varepsilon_{\chi_{0}}$,
$\chi_{0}(t)=1$ for $1-\varepsilon_{\chi_{0}}\leq t \leq 1$ and
$\chi_{0}'(t)\geq 0$ for $t\in [-1,1]$, where
$\varepsilon_{\chi_{0}}$ is a small positive number. In addition, we
define $\chi_{1}(t):=\chi_{0}'(t)$.

Consider $[-1,1] \times D^2$ equipped with a $1$-form
\begin{align*}\alpha=(1+\varepsilon
\chi_1(t)h)dt+\varepsilon((1-\chi_{0}(t))\beta_{0}+\chi_{0}(t)\beta_{1}).
\end{align*}
We then compute
\begin{align*}  d \alpha & =
\varepsilon (\chi_{1}(t)dh\wedge
dt+(1-\chi_{0}(t))d\beta_{0}+\chi_{0}(t)d\beta_{1}+\chi_{0}'(t)dt\wedge
\beta_{1} - \chi_{0}'(t)dt\wedge \beta_{0}) \\
& = \varepsilon(\chi_{1}(t)dh\wedge dt+\chi_{0}'(t)dt\wedge
\beta_{1}
- \chi_{0}'(t)dt\wedge \beta_{0})+\varepsilon\omega \\
& = \varepsilon(\chi_{1}(t)dh\wedge dt-\chi_{0}'(t)\beta_{1}\wedge
dt + \chi_{0}'(t)\beta_{0}\wedge dt)+\varepsilon\omega \\
& = \varepsilon(\chi_{1}(t)dh\wedge
dt-\chi_{0}'(t)(\beta_{1}-\beta_{0})\wedge dt)+\varepsilon\omega \\
& = \varepsilon(\chi_{1}(t)dh\wedge
dt-\chi_{0}'(t)dh\wedge dt)+\varepsilon\omega \\
& = \varepsilon(\chi_{1}(t) - \chi_{0}'(t))dh\wedge dt +
\varepsilon\omega = \varepsilon\omega
\end{align*} and hence
\begin{align*}
\alpha\wedge d\alpha& = ((1+\varepsilon
\chi_1(t)h)dt+\varepsilon((1-\chi_{0}(t))\beta_{0}+\chi_{0}(t)\beta_{1}))\wedge
\varepsilon\omega \\
& = \varepsilon dt\wedge \omega + \varepsilon^2(\chi_1(t)hdt +
(1-\chi_{0}(t))\beta_{0}+\chi_{0}(t)\beta_{1})\wedge \omega.
\end{align*}
If $\varepsilon$ is sufficiently small, then $\alpha$ satisfies the
contact condition, i.e., $\alpha \wedge d\alpha > 0$.

Now let us show that the Reeb vector field $R_{\alpha}$ is given by
\begin{align*} R_{\alpha}=\frac{1}{1+\varepsilon
\chi_1(t)h} \frac{\partial}{\partial t}.
\end{align*}
First we compute
\begin{align*}
i_{R_{\alpha}}(d\alpha) & = \left (\frac{1}{1+\varepsilon
\chi_1(t)h}
\partial_{t}\right )\lrcorner(\varepsilon \omega) = 0.
\end{align*}
Then we check the normalization condition, i.e.,
$\alpha(R_{\alpha})=1$:
\begin{align*}
\alpha(R_{\alpha}) & = ((1+\varepsilon
\chi_1(t)h)dt+\varepsilon((1-\chi_{0}(t))\beta_{0}+\chi_{0}(t)\beta_{1}))\left
(\frac{1}{1+\varepsilon \chi_1(t)h} \frac{\partial}{\partial
t}\right ) \\
& = \frac{1+\varepsilon \chi_1(t)h}{1+\varepsilon \chi_1(t)h} = 1.
\end{align*}
Since $\beta_1 = \beta_0$ in a neighborhood of $\partial D$, $h=0$
in a neighborhood of $\partial D^2$ and hence
$R_{\alpha}=\frac{\partial}{\partial t}$ in a neighborhood of
$[-1,1]\times \partial D^2$. Finally, we see that $\alpha$ satisfies
Conditions $(1)-(4)$.
\end{proof}

Fix $R_{\ast}\gg r_{sing}$ such that there is an annular
neighborhood $V_{R_{\ast}}$ of $\partial D(R_{\ast})$ in $\mathbb
R^2$ with $V_{R_{\ast}}\subset S$. Consider $D(R_{\ast})$ with two
$1$-forms $\beta_{0}:=\beta|_{D(R_{\ast})}$, where $\beta$ is a
$1$-form from Lemma~\ref{le1_formanalcon}, and
$\beta_{1}:=\varphi_{X_{H}}^{\ast}(\beta)|_{D(R_{\ast})}$. By
Remark~\ref{re1_analconstr},
\begin{align}\label{e1_cylinf}
\beta_{0}=\beta_{1}\quad \mbox{on}\quad V_{R_{\ast}}\cap
D(R_{\ast}).
\end{align}
In addition, we have
\begin{align}\label{e2_cylinf}
d\beta_{1}=d(\varphi_{X_{H}}^{\ast}(\beta)|_{D(R_{\ast})})=\varphi_{X_{H}}^{\ast}(d\beta)|_{D(R_{\ast})}=
(d\beta)|_{D(R_{\ast})}=d\beta_{0}>0.
\end{align}
From Equations~(\ref{e1_cylinf}) and (\ref{e2_cylinf}) it follows
that $\beta_{0}$ and $\beta_{1}$ satisfy the conditions of
Lemma~\ref{le1_almvertical}.

Now take $[-1,1]\times D(R_{\ast})$ with the contact $1$-form
$\alpha$ from Lemma~\ref{le1_almvertical} with $\beta_{0}$ and
$\beta_{1}$ as in the previous paragraph. Note that
$\beta_{1}-\beta_{0}=dh$ for $h\in C^{\infty}(D(R_{\ast}))$. We can
rewrite this equation as
\begin{align}\label{e3_cylinf}
\varphi_{X_{H}}^{\ast}(\beta)|_{D(R_{\ast})}-\beta|_{D(R_{\ast})}=dh.
\end{align}
Let us remind that
\begin{align*}
\varphi_{X_{H}}^{\ast}(\beta)-\beta=df
\end{align*}
on $\mathbb R^2$, where
\begin{align*}
f=\int^{1}_{0}(-H+\beta (X_{H}))\circ\varphi_{X_{H}}^{s}ds.
\end{align*}
Hence, $h:=f|_{D(R_{\ast})}$ satisfies Equation~(\ref{e3_cylinf}).
From Remark~\ref{re1_funcanalcon} it follows that
$f|_{D(R_{\ast})}=0$ on $(\cup^{n-1}_{k=1} V_{k})\cup
(D(R_{\ast})\cap S)$. Thus, $h=0$ on $(\cup^{n-1}_{k=1} V_{k})\cup
(D(R_{\ast})\cap S)$.
\begin{remark}
Since $h=0$ on $(\cup^{n-1}_{k=1} V_{k})\cup (D(R_{\ast})\cap S)$,
by the construction of $\alpha$, $R_{\alpha}=\partial_{t}$ on
$(\cup^{n-1}_{k=1} [-1,1]\times V_{k})\cup([-1,1]\times
(D(R_{\ast})\cap S))$.
\end{remark}
Let $\beta_{-}:=\varepsilon\beta_{0}$ and
$\beta_{+}:=\varepsilon\beta_{1}$, where $\varepsilon$ is a constant
from Lemma~\ref{le1_almvertical} which makes $\alpha$ contact.

\subsection{Gluing}
\label{section:gluing}

In this section we will construct the sutured contact solid torus
with $2n$ parallel longitudinal sutures, where $n\geq 2$.

First we construct surfaces with boundary $P_{+},P_{-}, D\subset
\mathbb R^2$ with the following properties:
\begin{itemize}
\item[(1)] $P_{\pm}\subset D$;
\item[(2)] $(\partial P_{\pm})_{\partial} \subset \partial D$ and $(\partial P_{\pm})_{int}\subset int(D)$;
\item[(3)] $\varphi_{X_{H}}$ maps $P_{+}$ to $P_{-}$ in such a way that $\varphi_{X_{H}}((\partial P_{+})_{int})=(\partial P_{-})_{\partial}$ and
$\varphi_{X_{H}}((\partial P_{+})_{\partial})=(\partial
P_{-})_{int}$;
\item[(4)] $(\partial P_{-})_{\partial} \cap (\partial P_{+})_{\partial} =
\emptyset$.
\end{itemize}

Recall that
\begin{align*}
X_{H}=n\mu r\sin(n\theta)\frac{\partial}{\partial r}+2\mu
\cos(n\theta)\frac{\partial}{\partial \theta}\quad \mbox{and}\quad
\beta_{-}=\beta_{+}=\frac{\varepsilon}{2}r^2d\theta
\end{align*}
on $D(R_{\ast})\cap S$. Note that $X_{H}$ is collinear to
$-\partial_{r}$ for $\theta=\frac{3\pi}{2n}+\frac{2\pi k}{n}$, where
$k\in \{0,\dots,n-1\}$. For simplicity, let us denote
\begin{align*}
\theta^{-}_{k}:=\frac{3\pi}{2n}+\frac{2\pi k}{n}-\frac{\pi}{2n}\quad
\mbox{and}\quad \theta^{+}_{k}:=\frac{3\pi}{2n}+\frac{2\pi
k}{n}+\frac{\pi}{2n},
\end{align*}
where $k\in\{ 0,\dots,n-1\}$.

Fix $R$ such that $r_{sing}\ll R \ll R_{\ast}$ and there is an
annular neighborhood $V(R)$ of $\partial D(R)$ in $D(R_{\ast})$
satisfying $\{(r,\theta)\ : \ R\leq r\leq R_{\ast} \}\subset
V(R)\subset S\cap D(R_{\ast})$.

Consider $D(R)\subset D(R_{\ast})$. Let $a^{+}_{k}$ be a segment on
$\partial D(R)$ which starts at $(R,\theta^{-}_{k})$ and ends at
$(R,\theta^{+}_{k})$, i.e., $a^{+}_{k}:=\{(R,\theta):\theta^{-}_{k}
\leq\theta\leq \theta^{+}_{k}\}$.

Consider $\cup_{k=0}^{n-1} \varphi_{X_{H}}(a^{+}_{k})$. It is easy
to see that every level set of $H$ which intersects $a^{+}_{k}$
intersects it only once. Hence, using that there are no closed level
sets of $H$ and $X_{H}$ is $\frac{2\pi}{n}$-symmetric on
$D(R_{\ast})\setminus D(r_{sing})$, we get
\begin{align*}
\left(\bigcup_{k=0}^{n-1} \varphi_{X_{H}}(a^{+}_{k})\right)\bigcap
\left(\bigcup_{k=0}^{n-1} a^{+}_{k}\right)=\emptyset.
\end{align*}
Let $a^{-}_{k}:=\varphi_{X_{H}}(a^{+}_{k})$. By possibly making $R$
and $R_{\ast}$ big enough, we can make $a^{-}_{k}$'s to be in
$V(R)$. Consider the endpoints of $a^{-}_{k}$'s. Since $X_{H}$ is
$\frac{2\pi}{n}$-symmetric outside of $D(r_{sing})$, it is easy to
see that
$\varphi_{X_{H}}(R,\theta^{-}_{k})=(\tilde{R},\tilde{\theta}^{-}_{k})$
and
$\varphi_{X_{H}}(R,\theta^{+}_{k})=(\tilde{R},\tilde{\theta}^{+}_{k})$,
where
$\theta^{-}_{k}-\frac{\pi}{2n}<\tilde{\theta}^{-}_{k}<\theta^{-}_{k}$,
$\theta^{+}_{k}<\tilde{\theta}^{+}_{k}<\theta^{+}_{k}+\frac{\pi}{2n}$
and $\tilde{R}>R$. In addition, observe that $\tilde{R}$ is the same
for all endpoints of $a^{+}_{k}$'s.

Let $\{b^{+}_{k}\}_{k=0}^{n-1}$ be a set of embedded curves on
$D(R_{\ast})$ with the following properties:
\begin{itemize}
\item[($P_{1}$)] $b^{+}_{k-1}$ starts
at the terminal point of $a^{+}_{k-1}$ and ends at the initial point
of $a^{+}_{k}$, where $k$ is considered mod $n$;
\item[($P_{2}$)] $b^{+}_{k-1}\subset \{ (r,\theta): r>r_{sing},\ \theta^{+}_{k-1} \leq \theta \leq
\theta^{-}_{k}\}$ and $b^{+}_{k-1}\subset V(R)$ for $k=0,\dots,n-1$;
\item[($P_{3}$)] $\varphi_{X_{H}}(b^{+}_{k})\subset \{(r,\theta): r> R\}\subset V(R)$ for $k=0,\dots,n-1$;
\item[($P_{4}$)] the region bounded by $a^{+}_{k}$'s and $b^{+}_{k}$'s
has smooth boundary;
\item[($P_{5}$)] each level set of $H$ which intersects $b^{+}_{k}$
intersects it only once.
\end{itemize}
For simplicity, we take $\frac{2\pi}{n}$-symmetric $b^{+}_{k}$'s,
i.e., $b^{+}_{k+1}$ can be obtained from $b^{+}_{k}$, where $k$ is
considered mod $n$, by doing $\frac{2\pi}{n}$-positive rotation
about the center of $D(R_{\ast})$.

\begin{figure}[t]
\begin{center} {\includegraphics [width=200pt, height=200pt]
    {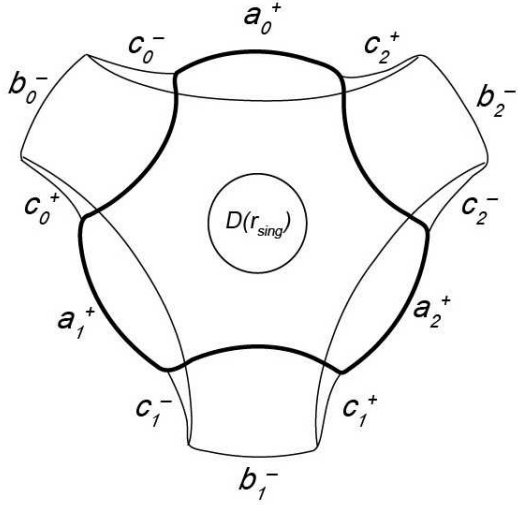}}
\end{center}
\caption{Construction of $P_{+}$, $P_{-}$ and $D$ in the case
$n=3$}\label{constofd}
\end{figure}

Note that Properties $(P_{2})$ and $(P_{3})$ and the form of $X_{H}$
on $D(R_{\ast})\setminus D(r_{sing})$ imply that
\begin{align}\label{e1_infgl}
\varphi_{X_{H}}(b^{+}_{k-1})\subset \{ (r,\theta)\ : \ r>R,\
\tilde{\theta}^{+}_{k-1} \leq \theta \leq
\tilde{\theta}^{-}_{k}\}\subset V(R),
\end{align}
where $k\in \{0,\dots,n-1\}$. Again, using that the level sets of
$H$ which intersects $b^{+}_{k}$ intersects it only once, there are
no closed level sets of $H$ and $X_{H}$ is
$\frac{2\pi}{n}$-symmetric, we obtain
\begin{align*}
\left(\bigcup_{k=0}^{n-1} \varphi_{X_{H}}(b^{+}_{k})\right)\bigcap
\left(\bigcup_{k=0}^{n-1} b^{+}_{k}\right)=\emptyset.
\end{align*}
Let $b^{-}_{k}:=\varphi_{X_{H}}(b^{+}_{k})$. From
Formula~(\ref{e1_infgl}) and the construction of $a^{+}_{k}$'s it
follows that
\begin{align}\label{e2_infgl}
\left(\bigcup_{k=0}^{n-1} a^{+}_{k}\right)\bigcap
\left(\bigcup_{k=0}^{n-1} b^{-}_{k}\right)=\emptyset.
\end{align}

Then we connect the terminal point of $a^{+}_{k}$ with the initial
point of $b^{-}_{k}$ by the line segment $c^{-}_{k}$ and the
terminal point of $b^{-}_{k}$ with the initial point of
$a^{+}_{k+1}$ by the line segment $c^{+}_{k}$. From the construction
above it follows that $c^{-}_{k}$ intersects $D(R)$ only at the
terminal point of $a^{+}_{k}$, and $c^{+}_{k}$ intersects $D(R)$
only at the initial point of $a^{+}_{k+1}$. Then we round the
corners between $c^{-}_{k}$ and $a^{+}_{k}$, $c^{-}_{k}$ and
$b^{-}_{k}$, $c^{+}_{k}$ and $b^{-}_{k}$, $c^{+}_{k}$ and
$a^{+}_{k+1}$. Finally, we get a surface whose boundary consists of
$a^{+}_{k}$'s, $b^{-}_{k}$'s, $c^{+}_{k}$'s and $c^{-}_{k}$'s, which
we call $D$. See Figure~\ref{constofd}.

\begin{remark}\label{re1_boundparts}
By the construction, $a^{\pm}_{k}$'s, $b^{\pm}_{k}$'s and
$c^{\pm}_{k}$ lie in $V(R)\subset D(R_{\ast})\cap S$.
\end{remark}

Now we take $[-1,1]\times D$ with a contact form
$\alpha:=\alpha|_{[-1,1]\times D}$ and contact structure
$\xi=ker(\alpha)$. Let $\Gamma=\{0\}\times \partial D$ in
$[-1,1]\times D$ and $U(\Gamma):=[-1,0]\times [-1,1]\times \Gamma$
be a neighborhood of $\Gamma$ with coordinates $(\tau,t)\in
[-1,0]\times [-1,1]$, where $t$ is a usual $t$-coordinate on
$[-1,1]\times D$. By Remark~\ref{re1_boundparts}, we can make
$U(\Gamma)$ such that
\begin{align}\label{e3_infgl}
U(\Gamma)\subset [-1,1]\times (D\cap S).
\end{align}
Observe that $\partial_{h}([-1,1]\times D)=R_{+}\cup R_{-}$, where
$R_{+}=\{-1\}\times D$ and $R_{+}=\{1\}\times D$ with respect to the
coordinates on $[-1,1]\times D$. In addition,
$\partial_{v}([-1,1]\times D)=[-1,1]\times
\partial D$ with respect the coordinates on $[-1,1]\times D$. Let $\beta_{\pm}:=\beta_{\pm}|_{\{\pm1\}\times
D}$.

\begin{lemma}\label{le1_suturbefglug}
$([-1,1]\times D, \Gamma, U(\Gamma), \xi)$ is a sutured contact
manifold and $\alpha$ is an adapted contact form.
\end{lemma}
\begin{proof}
First note that $\alpha|_{R{-}}=\beta_{-}$ and
$\alpha|_{R_{+}}=\beta_{+}$. Let us check that $(R_{-},\beta_{-})$
and $(R_{+},\beta_{+})$ are Liouville manifolds. From the
construction of $\beta_{\pm}$ it follows that
$d(\beta_{-})=d(\beta_{+})>0$. Since $\beta_{-}=\beta_{+}$ on $D\cap
S$, by Equation~(\ref{e3_infgl}), we have $\alpha=dt+\beta_{-}$ on
$U(\Gamma)$. Recall that
$\beta_{-}=\beta_{+}=\frac{\varepsilon}{2}r^2d\theta$ on $D\cap S$.
Hence, $\alpha|_{U(\Gamma)}=dt+\frac{\varepsilon}{2} r^2d\theta$.
The calculation
\begin{align*}
i_{Y_{\pm}|_{R_{\pm}\cap
U(\Gamma)}}(d\beta_{\pm})=\left(\frac{1}{2}r\partial_{r}\right)\lrcorner(\varepsilon
rdr\wedge d\theta)=\frac{\varepsilon}{2} r^2d\theta=\beta_{\pm}
\end{align*}
implies that the Liouville vector fields $Y_{\pm}|_{R_{\pm}\cap
U(\Gamma)}$ are equal to $\frac{1}{2}r\partial_{r}$. From the
construction of $D$ it follows that $Y_{\pm}$ is positively
transverse to $\partial R_{\pm}$. Therefore,
$(R_{-},\varepsilon\beta_{0})$ and $(R_{+},\varepsilon\beta_{1})$
are Liouville manifolds. As we already mentioned,
$\alpha=dt+\beta_{-}$ on $U(\Gamma)$. Finally, if we take $\tau$
such that $\partial_{\tau}=\frac{1}{2}r\partial_{r}$, then
$([-1,1]\times D, \Gamma, U(\Gamma), \xi)$ becomes a sutured contact
manifold with an adapted contact form $\alpha$.
\end{proof}

Now observe that from the construction of $\alpha$ it follows that
$\alpha|_{[-1,-1+\varepsilon_{\chi_{0}}]\times D}=dt+\beta_{-}$ and
$\alpha|_{[1-\varepsilon_{\chi_{0}},1]\times D}=dt+\beta_{+}$. Then
we define $P_{\pm}\subset R_{\pm}$. Let $P_{+}$ be a region bounded
by $a^{+}_{k}$'s and $b^{+}_{k}$'s in $R_{+}$ and let $P_{-}$ be a
region bounded by $a^{-}_{k}$'s and $b^{-}_{k}$'s in $R_{-}$. Note
that from Remark~\ref{re1_boundparts} it follows that
$a^{\pm}_{k}$'s and $b^{\pm}_{k}$ are in $S$. Hence, by
Lemma~\ref{le1_suturbefglug}, $Y_{\pm}=\frac{1}{2}r\partial_{r}$
along $\partial P_{\pm}$. The construction of $a^{\pm}_{k}$'s and
$b^{\pm}_{k}$ implies that $Y_{\pm}$ is positively transverse to
$\partial P_{\pm}$. From the construction we made it is easy to see
that $(\partial P_{+})_{\partial}=\cup_{k=0}^{n-1} a^{+}_{k}$,
$(\partial P_{+})_{int}=\cup_{k=0}^{n-1} b^{+}_{k}$, $(\partial
P_{-})_{\partial}=\cup_{k=0}^{n-1} b^{-}_{k}$ and $(\partial
P_{-})_{int}=\cup_{k=0}^{n-1} a^{-}_{k}$. If $\pi\co[-1,1]\times D
\to D$ is a projection to $D$ along $\partial_{t}$, then from
Equation~(\ref{e2_infgl}) it follows that $\pi((\partial
P_{-})_{\partial})\cap \pi ((\partial P_{+})_{\partial})=\emptyset$.
Observe that
$\varphi_{X_{H}}^{\ast}(\beta_{-}|_{P_{-}})=\beta_{+}|_{P_{+}}$ and
$\varphi_{X_{H}}(a^{+}_{k})=a^{-}_{k}$,
$\varphi_{X_{H}}(b^{+}_{k})=b^{-}_{k}$ for $k=0,\dots,n-1$. Hence,
by definition of $P_{\pm}$, $\varphi_{X_{H}}$ sends $P_{+}$ to
$P_{-}$ in such a way that $(\partial P_{+})_{int}$ maps to
$(\partial P_{-})_{\partial}$ and $(\partial P_{+})_{\partial}$ maps
to $(\partial P_{-})_{int}$.

Next, we follow the gluing procedure overviewed in
Section~\ref{section:glsutcontman} and completely described in
\cite{ColinGhigginiHondaHutchings}. We get a sutured contact
manifold $(S^1\times D^2, \tilde{\Gamma},
U(\tilde{\Gamma}),\tilde{\alpha})$. For simplicity, we omit index
$N$. Observe that the region enclosed by $\partial D$ and
$a^{-}_{k}$ in $D\setminus P_{-}$ contains $a^{+}_{k}$ and the
region enclosed by $\partial D$ and $b^{+}_{k}$ in $D\setminus
P_{+}$ contains $b^{-}_{k}$. Then from the gluing construction and
the form of $\varphi_{X_{H}}$ near the boundary of $P_{+}$ it
follows that $\tilde{\Gamma}$ has $2n$ parallel longitudinal
components.

\subsection{Reeb orbits}
\label{section:reeb_orb}

Consider $(S^1\times D^2, \tilde{\Gamma},
U(\tilde{\Gamma}),\tilde{\alpha})$ obtained in
Section~\ref{section:gluing}. Recall that $\tilde{\Gamma}$ consists
of $2n$ parallel longitudinal curves. Let $\tilde{\xi}$ denote the
contact structure defined by $\tilde{\alpha}$ and
$R_{\tilde{\alpha}}$ denote the Reeb vector field defined by
$\tilde{\alpha}$. The main goal of this section is to understand the
set of embedded, closed orbits of $R_{\tilde{\alpha}}$.

\begin{definition}
Let $S$ be a non-empty set with two non-empty subsets $S_{1}$ and
$S_{2}$ such that $S_{1}\cap S_{2}\ne \emptyset$, and let $f\co
S_{1}\to S_{2}$. A point $s\in S_{1}$ is called a {\em periodic
point} of $f$ of period $n$ if $f^n(s)$ is well-defined, i.e.,
$f^{i}(s)\in S_{1}\cap S_{2}$ for $i=1,\dots,n-1$, and $f^n(s)=s$.
\end{definition}

\begin{lemma}\label{le1_anconstnumboforb}
$R_{\tilde{\alpha}}$ has $n-1$ embedded, closed orbits.
\end{lemma}
\begin{proof}
First consider $\varphi_{X_{H}}|_{P_{+}}$. Recall that from the
construction of $P_{-}$ and $P_{+}$ it follows that
$\varphi_{X_{H}}(P_{+})=P_{-}$. Hence, by
Remark~\ref{re1_phiperpoints}, $\{ p_{k} \}^{n-1}_{k=1}$ is the set
of periodic points of $\varphi_{X_{H}}|_{P_{+}}$.

From the construction of $\alpha$ on $[-1,1]\times D$ and the gluing
construction it follows that there is a one-to-one correspondence
between the set of embedded Reeb orbits and the set of periodic
points of $\varphi_{X_{H}}|_{P_{+}}$. Thus, there are $n-1$ embedded
closed orbits of $R_{\tilde{\alpha}}$.
\end{proof}

Let $\gamma_{k}$ be the embedded, closed orbit, which corresponds to
the periodic point $p_{k}$, i.e., $\gamma_{k}$ is obtained from
$[-1,1]\times \{p_{k}\}\subset [-1,1]\times D$.

\begin{lemma}\label{le1_anconstsymplact}
$\gamma^{s}_{k}$ is a nondegenerate orbit for $k\in\{1,\dots,n-1\}$
and $s\in \mathbb N$. Moreover $\{\gamma_{k}\}^{n-1}_{k=1}$ is a set
of positive hyperbolic orbits and
$\int\limits_{\gamma_{l}}\tilde{\alpha}=\int\limits_{\gamma_{m}}\tilde{\alpha}$
for $l,m=1,\dots,n-1$.
\end{lemma}

\begin{proof}
Let
\begin{align*}
M^{(0)}=(([-1,1]\times D)\cup (R_{+}(\Gamma)\times[1;\infty))\cup
(R_{+}(\Gamma)\times (-\infty;-1]))
\end{align*}
and
\begin{align*}
\tilde{M}=M^{(0)} \setminus ((P_{+}\times (N, \infty)\cup
(P_{-}\times (-\infty,-N)).
\end{align*}

In addition, let $\alpha_{\tilde{M}}$ denote the contact form on
$\tilde{M}$ and let $\xi_{\tilde{M}}$ denote the contact structure
defined by $\alpha_{\tilde{M}}$.

Consider $[-1,1]\times D\subset \tilde{M}$. From the construction of
$\alpha$ it follows that $\alpha|_{[-1,1]\times
V_{k}}=dt+\beta_{-}$. Since the contact structure on
$[1,\infty)\times P_{+}$ is given by $dt+\beta_{+}$ and the contact
structure on $(-\infty, -1]\times P_{-}$  is given by
$dt+\beta_{-}$, $\alpha_{\tilde{M}}|_{[-N,N]\times
V_{k}}=dt+\beta_{-}$ on $[-N,N]\times V_{k}\subset \tilde{M}$.
Therefore, we get
\begin{align}\label{e10_anconstlength}
\int\limits_{[-N,N]\times \{p_{k}\}} \alpha_{\tilde{M}}=2N.
\end{align}
From the gluing construction and Equation~(\ref{e10_anconstlength})
it follows that $\int_{\gamma_{k}} \tilde{\alpha}=2N$. Note that
$\int_{\gamma_{k}} \tilde{\alpha}$ does not depend on $k$. Hence,
$\int_{\gamma_{l}}\tilde{\alpha}=\int_{\gamma_{m}}\tilde{\alpha}$
for $l,m=1,\dots,n-1$.

Now observe that $H|_{V_{k}}=axy$ and hence
\begin{align*}\varphi_{X_{H}}|_{V_{k}}=
\left ( \begin{array}{ll} \lambda & 0\\
0 & \lambda^{-1}
\end{array}
\right ),
\end{align*}
where $\lambda=e^{a}$. Let the symplectic trivialization of
$\xi_{\tilde{M}}$ along $[-N,N]\times \{p_{k}\}$ be given by the
framing
$(\lambda^{\frac{-N-t}{2N}}\partial_{x},\lambda^{\frac{t+N}{2N}}\partial_{y})$.
Note that the symplectic trivialization of $\xi_{\tilde{M}}$ gives
rise to the symplectic trivialization of $\tilde{\xi}$ along
$\gamma_{k}$.

It is easy to see that the linearized return map $P_{\gamma_{k}}$ is
given by
\begin{align*}
P_{\gamma_{k}}=\left ( \begin{array}{ll} \lambda & 0\\
0 & \lambda^{-1}
\end{array}
\right ).
\end{align*}
Since the eigenvalues of $P_{\gamma_{k}}$ are positive real numbers
different from $1$, $\gamma_{k}$ is a positive hyperbolic orbit.
Hence, $\{\gamma_{k}\}^{n-1}_{k=1}$ is a set of positive hyperbolic
orbits of $R_{\tilde{\alpha}}$. In addition,
$P_{\gamma^{s}_{k}}=P^s_{\gamma_{k}}$. Therefore, the eigenvalues of
$P_{\gamma^{s}_{k}}$ are different from $1$. Hence, $\gamma^{s}_{k}$
is a nondegenerate orbit for $s\in \mathbb N$.

\end{proof}

\section{Calculations}
\label{calculations}

In this section we will calculate the sutured embedded contact
homology, the sutured cylindrical contact homology and the sutured
contact homology of the sutured contact solid torus constructed in
Section~\ref{section:gluing}.

Consider the symplectization $(\mathbb R\times (S^1\times
D^2)^{\ast}, d(e^s\tilde{\alpha}^{\ast}))$ of $((S^1\times
D^2)^{\ast}, \tilde{\alpha}^{\ast})$, where $s$ is the coordinate on
$\mathbb R$ and $((S^1\times D^2)^{\ast}, \tilde{\alpha}^{\ast})$ is
the completion of $(S^1\times D^2,\tilde{\Gamma}, U(\tilde{\Gamma}),
\tilde{\alpha})$. Let $J$ be an almost complex structure on
$(\mathbb R\times (S^1\times D^2)^{\ast},
d(e^s\tilde{\alpha}^{\ast}))$ tailored to $((S^1\times D^2)^{\ast},
\tilde{\alpha}^{\ast})$.

\subsection{Sutured embedded contact homology}
\label{section:sut_emb_cont_hom}

Consider the set of embedded, closed orbits of $R_{\tilde{\alpha}}$.
By Lemma~\ref{le1_anconstnumboforb}, $R_{\tilde{\alpha}}$ has $n-1$
embedded, closed orbits $\gamma_{1},\dots,\gamma_{n-1}$, which are
positive hyperbolic by Lemma~\ref{le1_anconstsymplact}. In addition,
Lemma~\ref{le1_anconstsymplact} implies that all Reeb orbits are
nondegenerate. From the gluing construction, i.e., since
$\{p\}^{n-1}_{k=1}$ is a set of fixed points of $\varphi_{X_{H}}$,
it follows that $[\gamma_{i}]$ is a generator of $H_{1}(S^1\times
D^2; \mathbb Z)$ for $i\in \{1,\dots,n-1\}$ and
$[\gamma_{i}]=[\gamma_{j}]$  for $i,j\in \{1,\dots,n-1\}$. From now
on we identify $H_{1}(S^1\times D^2; \mathbb Z)$ with $\mathbb Z$ in
such a way that $[\gamma_{i}]\in H_{1}(S^1\times D^2; \mathbb Z)$ is
identified with $1 \in \mathbb Z$ for $i\in \{ 1,\dots,n-1\}$.
Recall that multiplicities of hyperbolic orbits in an admissible
orbit set must be equal to $1$. Hence, from
Lemma~\ref{le1_anconstsymplact} it follows that the admissible orbit
sets are of the form $\{ (\gamma_{i_1},1),\dots, (\gamma_{i_s},1)
\}$, where $1\leq i_1<\dots<i_s \leq n-1$. Note that $\emptyset$ is
an admissible orbit set. For ease of notation, we write
$\gamma_{i_1}\dots\gamma_{i_s}$ instead of $ \{(\gamma_{i_1},
1),\dots,(\gamma_{i_s}, 1) \}$ and $1$ instead of $\emptyset$, where
$1\leq i_1<\dots<i_s\leq n-1$.

\begin{lemma}\label{le1_anconstechdiff}
Let $\partial$ be the ECH differential. Then $\partial(a)=0$ for
every admissible orbit set $a$.
\end{lemma}
\begin{proof}
Fix $h\in H_{1}(S^1\times D^2; \mathbb Z)$. Let $S_{h}$ be a set of
admissible orbit sets with homology class $h$. It is easy to see
that
\begin{align*}
S_{h}=\left \{
\begin{array}{ll}
\{ \gamma_{i_1}\dots\gamma_{i_h} \}, & \mbox{for}\ 0\leq h \leq n-1;\\
\emptyset, & \mbox{otherwise}.
\end{array}
\right.
\end{align*}
From Lemma~\ref{le1_anconstsymplact} it follows that for every $a\in
S_{h}$, $\mathcal A(a)=2Nh$. Let $a,b\in S_{h}$ be different
admissible orbit sets. Then, as we already mentioned,
\begin{align}\label{e11_anconstdifforbss}
\mathcal A(a)=\mathcal A(b)=2Nh.
\end{align}
From Equation~(\ref{e11_anconstdifforbss}) and the second part of
Lemma~\ref{le1_symplactfech} it follows that $M^{J}(a,b)$ is empty.
In addition, by the second part of Lemma~\ref{le1_symplactfech},
every element in $M^{J}(a,a)$ maps to a union of trivial cylinders.
Hence, by Proposition~\ref{pr1_echind} and definition of $\partial$,
$\partial(a)=0$. Note that trivial cylinders are regular and hence
we can omit the genericity assumption in
Proposition~\ref{pr1_echind}. Thus, for every admissible orbit set
$a$, $\partial(a)=0$.
\end{proof}

Again, let $S_{h}$ be a set of admissible orbit sets with homology
class $h$.

By counting the number of element in $S_{h}$, we get
\begin{align}\label{e9_acech}
|S_{h}|=\left \{
\begin{array}{ll}
{n-1 \choose h}, & \mbox{for}\ 0\leq h \leq n-1;\\
0, & \mbox{otherwise}.
\end{array}
\right.
\end{align}
By Equation~(\ref{e9_acech}) and Lemma~\ref{le1_anconstechdiff}, we
get
\begin{align*}
ECH(S^1\times D^2,\tilde{\Gamma},\tilde{\alpha},J,h)\simeq
\Lambda^{\ast}\langle\gamma_1,\dots,\gamma_{n-1}\rangle\simeq \left
\{
\begin{array}{ll}
\mathbb Z^{n-1 \choose h}, & \mbox{for}\ 0\leq h \leq n-1;\\
0, & \mbox{otherwise}.
\end{array}
\right.
\end{align*}
Here $\Lambda^{\ast}\langle\gamma_1,\dots,\gamma_{n-1}\rangle$ is
the exterior algebra over $\mathbb Z$ generated by
$\gamma_{1},\dots,\gamma_{n-1}$. Thus, we obtain
\begin{align*}
ECH(S^1\times D^2,\tilde{\Gamma},\tilde{\alpha}, J)=
\bigoplus\limits_{h\in H_{1}(S^1\times D^2; \mathbb Z)}
ECH(S^1\times D^2,\tilde{\Gamma},\tilde{\alpha}, J, h)\simeq \mathbb
Z^{\sum\limits_{h=0}^{n-1}{n-1 \choose h}}=\mathbb Z^{2^{n-1}}.
\end{align*}
This completes the proof of Theorem~\ref{suturedECH}.

\begin{remark}
Note that for the constructed sutured contact solid torus, the
sutured Floer homology coincides with the sutured embedded contact
homology. In fact, they agree in each Spin$^c$-structure. This
follows from Proposition 9.2 in \cite{Juh'asz3}, where the sutured
Floer homology of every sutured manifold $(S^1\times D^2, \Gamma)$
has been computed by Juh\'{a}sz.
\end{remark}

\subsection{Sutured cylindrical contact homology}
\label{section:sut_cyl_cont_hom}

First recall that Lemma~\ref{le1_anconstnumboforb} implies that all
closed orbits of $R_{\tilde{\alpha}}$ are nondegenerate.

\begin{remark}\label{cch_indep}
Note that there are no contractible Reeb orbits. Hence, from
Fact~\ref{welldefcylconthom}, Remark~\ref{actchomcchom} and the fact
that $\pi_{1}(S^1\times D^2; \mathbb Z)\simeq H_{1}(S^1\times D^2;
\mathbb Z)\simeq \mathbb Z$ it follows that for all $h\in
H_{1}(S^1\times D^2; \mathbb Z)$ $HC^{cyl, h}_{\ast}(S^1\times
D^2,\tilde{\Gamma},\tilde{\alpha}_{\delta}, J)$ is defined, i.e.,
$\partial^2=0$, and is independent of contact form $\tilde{\alpha}$
for the given contact structure $\tilde{\xi}$ and the almost complex
structure $J$.
\end{remark}

Note that $C_{\ast}(\tilde{\alpha}, J)$ splits as
\begin{align*}
C_{\ast}(\tilde{\alpha}, J)=\bigoplus\limits_{h\in H_{1}(S^1\times
D^2; \mathbb Z)} C_{\ast}^{h}(\tilde{\alpha}, J).
\end{align*}

From Lemma~\ref{le1_anconstsymplact} it follows that
$\{\gamma_{k}\}^{n-1}_{k=1}$ is a set of positive hyperbolic orbits.
Hence, the definition of the Conley-Zehnder index implies that
$\mu_{\tau}(\gamma^{s}_{l})$ is even for $l\in \{1,\dots,n-1\}$ and
$s\in \mathbb N$. Then, according to the definition of a good orbit,
it follows that $\gamma^{s}_{l}$ is a good orbit for $l\in
\{1,\dots,n-1\}$ and $s\in \mathbb N$. Hence, we get
\begin{equation}\label{e15_accch}
C_{\ast}^{h}(\tilde{\alpha}, J)=\left \{
\begin{array}{ll}
\mathbb Q \langle \gamma_{1}^{h},\dots,\gamma_{n-1}^{h}
\rangle, & \mbox{for}\ h\geq 1;\\
0, & \mbox{otherwise}.
\end{array}
\right.
\end{equation}
Here $\mathbb Q \langle \gamma_{1}^{h},\dots,\gamma_{n-1}^{h}
\rangle$ is a $\mathbb Q$-module freely generated by
$\gamma_{1}^{h},\dots,\gamma_{n-1}^{h}$. Now recall that
Lemma~\ref{le1_anconstsymplact} says that
$\int_{\gamma_{l}}\tilde{\alpha}=\int_{\gamma_{m}}\tilde{\alpha}=2N$,
where $l,m\in \{1,\dots,n-1\}$. Therefore,
\begin{align}\label{e21_anconsymact}
\int\limits_{\gamma^{s}_{l}}\tilde{\alpha}=\int\limits_{\gamma^{s}_{m}}\tilde{\alpha}=2Ns,
\end{align}
for $l,m\in \{1,\dots,n-1\}$ and $s\in \mathbb N$.

Remark~\ref{actchomcchom}, Equation~(\ref{e21_anconsymact}) and
definition of the cylindrical contact homology differential imply
that $\partial|_{C_{\ast}^{h}(\tilde{\alpha}, J)}=0$ for all $h\in
H_{1}(S^1\times D^2; \mathbb Z)$. Thus, using
Equation~(\ref{e15_accch}), we obtain
\begin{align}\label{e22_anconstcch}
HC^{cyl,h}(S^{1}\times D^2,\tilde{\Gamma},\tilde{\xi})&=
\bigoplus\limits_{i\in \mathbb Z} HC^{cyl,h}_{i}(S^1\times
D^2,\tilde{\Gamma},\tilde{\xi})\nonumber\\
&\simeq \left \{
\begin{array}{ll}
\mathbb Q^{n-1}, & \mbox{for}\ h\geq 1;\\
0, & \mbox{otherwise}.
\end{array}
\right.
\end{align}
Finally, Equation~(\ref{e22_anconstcch}) implies that
\begin{align*}
HC^{cyl}(S^{1}\times D^2,\tilde{\Gamma},\tilde{\xi})=
\bigoplus\limits_{h\geq 1}\bigoplus\limits_{i\in \mathbb Z}
HC^{cyl,h}_{i}(S^1\times D^2,\tilde{\Gamma},\tilde{\xi})\simeq
\bigoplus\limits_{h\geq 1}\mathbb Q^{n-1}.
\end{align*}
This completes the proof of Theorem~\ref{suturedCylContHom}.

\subsection{Sutured contact homology}
\label{section:sut_cont_hom}

Recall that from Lemma~\ref{le1_anconstnumboforb} it follows that
all closed orbits of $R_{\tilde{\alpha}}$ are nondegenerate. From
the discussion in the previous section it follows that
$\gamma^{s}_{l}$ is a good orbit for $l\in \{1,\dots,n-1\}$ and
$s\in \mathbb N$. Hence, the supercommutative algebra $\mathcal
A(\tilde{\alpha}, J)$ is generated by $\gamma^{s}_{l}$ for $l\in
\{1,\dots,n-1\}$ and $s\in \mathbb N$. Note that $\mathcal
A(\tilde{\alpha}, J)$ splits as
\begin{align*}
\mathcal A(\tilde{\alpha}, J)=\bigoplus\limits_{h\in H_{1}(S^1\times
D^2; \mathbb Z)} \mathcal A^{h}(\tilde{\alpha}, J),
\end{align*}
where $\mathcal A^{h}(\tilde{\alpha}, J)$ is generated, as a vector
space over $\mathbb Q$, by monomials of total homology class $h$.
Hence,
\begin{align*}
\mathcal A^{h}(\tilde{\alpha}, J)\simeq \mathbb Q^{\rho(n,h)},
\end{align*}
where $p(n,h)$ denotes the coefficient of $x^{h}$ in the generating
function $\prod^{\infty}_{s=1}(1+x^{s})^{n-1}$.

In \cite[Corollary 4.2]{Fabert}, Fabert proved that the differential
in contact homology and rational symplectic field theory is strictly
decreasing with respect to the symplectic action filtration. In
other words, branched covers of trivial cylinders do not contribute
to contact homology and rational symplectic  field theory
differentials.

From Lemma~\ref{le1_anconstsymplact} it follows that all generators
of $\mathcal A^{h}(\tilde{\alpha}, J)$ have the same symplectic
action and hence $\partial|_{\mathcal A^{h}(\tilde{\alpha}, J)}=0$
for all $h\in H_{1}(S^1\times D^2; \mathbb Z)$. Thus,
\begin{align*}
HC^{h}(S^1\times D^2, \tilde{\Gamma}, \tilde{\xi})=
\bigoplus\limits_{i\in \mathbb Z} HC^{h}_{i}(S^1\times
D^2,\tilde{\Gamma},\tilde{\xi}) &\simeq \mathbb Q^{\rho(n,h)}
\end{align*}
and hence
\begin{align*}
HC(S^{1}\times D^2,\tilde{\Gamma},\tilde{\xi})=
\bigoplus\limits_{h\in \mathbb Z}\bigoplus\limits_{i\in \mathbb Z}
HC^{h}_{i}(S^1\times D^2,\tilde{\Gamma},\tilde{\xi})\simeq
\bigoplus\limits_{h\in \mathbb Z}\mathbb Q^{\rho(n,h)}.
\end{align*}
This completes the proof of Theorem~\ref{suturedContHom}.


\begin{thebibliography}{}


\bibitem{Bourgeois}
F Bourgeois, \emph{A survey of Contact Homology}, lectures at
Yashafest, 2007.

\bibitem{BourgeoisEliashbergHoferWysockiZehnder}
F Bourgeois, Y Eliashberg, H Hofer, K Wysocki and E Zehnder,
\emph{Compactness results in symplectic field theory}, Geom. and
Top. {\bf 7} (2003), 799--888.

\bibitem{ColinGhigginiHondaHutchings}
V Colin, P Ghiggini, K Honda and M Hutchings, \emph{Sutures and
contact homology I}, preprint 2010, arXiv:1004.2942.

\bibitem{ColinHonda}
V Colin and K Honda, \emph{Constructions contr\^{o}l\'{e}es de
champs de Reeb et applications}, Geom. Topol. {\bf 9} (2005),
2193--2226.

\bibitem{Cotton-Clay}
A Cotton-Clay, \emph{Symplectic Floer homology of area-preserving
surface diffeomorphisms}, Geom. Topol. {\bf 13} (2009), 2619--2674.

\bibitem{EliashbergGiventalHofer}
Y Eliashberg, A Givental and H Hofer, \emph{Introduction to
symplectic field theory}, Geom. Funct. Anal. Special Volume {\bf 10}
(2000), 560--673.

\bibitem{Fabert}
O Fabert, \emph{Obstruction bundles over moduli spaces with boundary
and the action filtration in symplectic field theory}, preprint
2007, arXiv:0709.3312.

\bibitem{Gabai}
D Gabai, \emph{Foliations and the topology of 3-manifolds}, J. Diff.
Geom {\bf 18} (1983), 445--503.

\bibitem{Hutchings}
M Hutchings, \emph{An index inequality for embedded
pseudoholomorphic curves in symplectizations}, J. Eur. Math. Soc.
{\bf 4} (2002), 313--361.

\bibitem{HutchingsSullivan}
M Hutchings and M Sullivan, \emph{Rounding corners of polygons and
the embedded contact homology of $T^3$}, Geom. Topol. {\bf 10}
(2006), 169--266.

\bibitem{HutchingsTaubes}
M Hutchings and C\,H Taubes, \emph{Gluing pseudoholomorphic curves
along branched covered cylinders I}, J. Symplectic Geom. {\bf 5}
(2007), 43--137.

\bibitem{HutchingsTaubes2}
M Hutchings and C\,H Taubes, {\it Gluing pseudoholomorphic curves
along branched covered cylinders II}, J. Symplectic Geom. {\bf 7}
(2009), 29--133.

\bibitem{Juh'asz1}
A Juh\'{a}sz, \emph{Holomorphic disks and sutured manifolds},
Algebr. Geom. Topol. {\bf 6} (2006), 1429--1457 (electronic).

\bibitem{Juh'asz2}
A Juh\'{a}sz, \emph{Floer homology and surface decompositions},
Geom. Topol. {\bf 12} (2008), 299--350 (electronic).

\bibitem{Juh'asz3}
A Juh\'{a}sz, \emph{The sutured Floer homology polytope}, Geom.
Topol. {\bf 14} (2010), 1303--1354 (electronic).

\bibitem{KronheimerMrowka}
P\,B Kronheimer and T\,S Mrowka, \emph{Knots, sutures and excision},
preprint 2008, arXiv:0807.4891.

\bibitem{KronheimerMrowka2}
P\,B Kronheimer and T\,S Mrowka, \emph{Monopoles and
three-manifolds}, Cambridge University Press, 2008.

\bibitem{OzsvathSzabo}
P Ozsv{\'a}th and Z Szab{\'o}, \emph{Holomorphic disks and
topological invariants for closed
  three-manifolds}, Ann. of Math. (2) {\bf 159} (2004), 1027--1158.

\bibitem{Taubes1}
C\,H Taubes, \emph{Embedded contact homology and Seiberg-Witten
Floer homology I-IV}, preprints 2008.

\bibitem{Taubes2}
C\,H Taubes, \emph{The Seiberg-Witten equations and the Weinstein
conjecture}, Geom. Topol. {\bf 11} (2007), 2117--2202.

\end{thebibliography}
\end{document}